\documentclass[10pt, twoside]{amsart}
\usepackage{amssymb, amsmath, hyperref}

\newtheorem{theo}{Theorem}[section]
\newtheorem{coro}[theo]{Corollary}

\newtheorem{prop}[theo]{Proposition}

\theoremstyle{definition}
\newtheorem{defi}[theo]{Definition}
\newtheorem{rmk}[theo]{Remark}

\newtheorem{case}{Case}
\newtheorem*{caseA}{Case A}
\newtheorem*{caseB}{Case B}
\numberwithin{equation}{section}

\newcommand{\Z}{{\mathbb Z}}
\newcommand{\C}{{\mathbb C}}
\newcommand{\Q}{{\mathbb Q}}

\newcommand{\M}{{\mathcal M}}
\newcommand{\E}{{\mathcal E}}
\renewcommand{\S}{{\mathcal S}}
\renewcommand{\H}{{\mathcal H}}
\renewcommand{\O}{{\mathcal O}}
\newcommand{\G}{{\Gamma}}
\renewcommand{\a}{{\alpha}}
\renewcommand{\b}{{\beta}}
\newcommand{\g}{{\gamma}}
\renewcommand{\l}{{\lambda}}
\renewcommand{\o}{{\omega}}
\renewcommand{\i}{{\infty}}
\newcommand{\p}{{\prime}}
\newcommand{\vp}{{\varphi}}
\renewcommand{\sl}{{\textrm{SL}_2(\mathbb{Z})}}

\newcommand{\m}{{\begin{pmatrix}a & b\\ c & d \end{pmatrix}}}
\newcommand{\Qbar}{{\overline{\mathbb{Q}}}}
\newcommand{\Nmid}{{\hspace{.5mm} \nmid \hspace{.5mm} }}
\DeclareMathOperator{\Gal}{Gal}
\DeclareMathOperator{\ord}{ord}
\DeclareMathOperator{\re}{Re}

\newcommand{\z[2]}{{\zeta \rule[-.9mm]{0mm}{0mm}_{#2}}}
\newcommand{\ch[2]}{{\chi\rule[-.8mm]{0mm}{0mm}_{#2}}}
\newcommand{\one[2]}{{\bf 1}\rule[-.6mm]{0mm}{0mm}_{\hspace{-.2mm} #2}}
\newcommand{\Mod[2]}{{\; \ (\textrm{mod }#2)}}
\newcommand{\A}{{r}}
\newcommand{\B}{{s}}
\newcommand{\lp}{\ensuremath{(}}
\newcommand{\rp}{\ensuremath{)}}
\renewcommand{\.}{{\,}}
\newcommand{\ra}{{\rule[-3.1mm]{0mm}{7mm}}}
\newcommand{\rb}{{\rule[-2mm]{0mm}{6.5mm}}}

\begin{document}
\title{Hecke Eigenforms As Products Of Eigenforms}
\author{Matthew L. Johnson}
\address{Department of Mathematics \\ The University of Arizona}
\email{johnsoma@math.arizona.edu}
 
\begin{abstract}
We investigate when the product of two Hecke eigenforms for $\G_1(N)$ is again a Hecke eigenform. In this paper we prove that the product of two normalized eigenforms for $\G_1(N)$, of weight greater than $1$, is an eigenform only 61 times, and give a complete list. Duke \cite{duke99} and Ghate \cite{ghate00} independently proved that with eigenforms for $\sl$, there are only 16 such product identities. Ghate \cite{ghate02} also proved related results for $\G_1(N)$, with $N$ square free. Emmons \cite{emmons05} proved results for eigenforms away from the level, for prime level. The methods we use are elementary and effective, and do not rely on the Rankin-Selberg convolution method used by both Duke and Ghate.
\end{abstract}
\maketitle

\section{Introduction} \label{secc1}
In this paper we consider the following question: when does the equation
\[
h=f\cdot g
\]
have solutions with $f$, $g$ and $h$ Hecke eigenforms with respect to $\G_1(N)$, for any level $N$? This question was independently considered by Duke \cite{duke99} and Ghate \cite{ghate00} for eigenforms with respect to $\sl $ (i.e., $N=1$). They proved that there are only 16 such identities. 
Ghate \cite{ghate02} also proved that for $N$ square-free and $f$, $g \in \G_1(N)$ eigenforms of weight greater than $2$, the product $fg$ cannot be an eigenform when certain conditions on $f$ and $g$ are satisfied (note that Ghate uses a slightly more general definition of eigenforms than we do). Emmons \cite{emmons05} classified all products $h=fg$, with $f,g$, and $h$ eigenforms with respect to $\G_0(p)$ away from the level $p$, with $p \geq 5$ prime.

In this paper we extend Ghate's result, by removing the condition that the level be square-free, and by also considering weight $2$ eigenforms. We prove that for eigenforms $f$, $g$ and $h$ of level $N$ (with no restrictions on $N$) and $f$ and $g$ each having weight greater than $1$, there are only 61 eigenform products of the form $h=fg$. A complete list is given in Table \ref{table6}. For additional observations see Remark \ref{remm2}. In the process we give a new proof of the level $1$ case, which was previously addressed by both Duke and Ghate.

\begin{rmk}
Given an eigenform $f$ of level $N$, $f$ will also be an eigenform of level $M$ for any $M$ such that $N|M$ and $p|M \Rightarrow p|N$, for $p$ prime. Therefore an eigenform product identity $f g =h $ for level $N$ will also exist for each such level $M$ above $N$. We count the identity only once, for the smallest level it occurs. In addition, we regard all scalar multiples of an identity as the same identity, and count them only once.
\end{rmk}

\begin{theo} \label{thmm1}
There are only 61 tuples $(N ,k ,l ,\psi ,\vp )$, $l \geq k>1$, such that there exist Hecke eigenforms $f \in \M_k (N ,\psi )$ and $g \in \M_l (N ,\vp )$ with $fg$ also an eigenform. When $l>k$ there is one eigenform identity per tuple, and for all but four of the 18 tuples occurring with $l=k$ this is true as well. In each of the four exceptions, there are two eigenform identities associated to the two tuples $(N ,k ,k ,\vp ,\psi )$ and $(N ,k ,k ,\psi ,\vp )$. This results in a total of 61 eigenform product identities. A complete list of all identities is given in Table \ref{table6}.
\end{theo}
Hence the product of two eigenforms with respect to $\G_1(N)$, each of weight greater than 1, is \emph{never} an eigenform, save for 61 exceptions.

\begin{coro} 
There are only nine identities of the form $h=fg$, with $f$, $g$ and $h$ Hecke eigenforms of level $N>1$ and with the weights of $f$ and $g$ greater than $2$. These identities are given in Table \ref{table8} \lp{}in addition to being distributed among the 61 identities in Table \ref{table6}\rp{}. These nine identities occur at levels $3$, $4$ and $5$ and with the weight of $f$ and $g$ either $3$, $4$ or $5$. In addition there are 16 such identities at level $N=1$. They are listed in Table \ref{table4} and Table \ref{table5}.
\end{coro}

Our methods are elementary and effective, and do not rely on the Rankin-Selberg convolution that both Duke and Ghate use. Instead, we make extensive use of the multiplicative properties the Fourier coefficients of a Hecke eigenform must satisfy. Given an eigenform identity $f g =h $, either $f $ or $g $ must have a non-zero constant term, and therefore must be an Eisenstein series. The Fourier expansion of an Eisenstein series $f $ is easily determined by level, weight and Nebentypus. The expansion of a cuspidal eigenform, however, is not so easily determined. Using Proposition \ref{propp3} we are able to restrict our attention to a finite number of Eisenstein eigenforms. Given a fixed Eisenstein eigenform $f $, we use the multiplicative properties of the coefficients of $g $ and $f g $ to \emph{uniquely} determine the expansion of $g $ term by term. We then examine if this expansion corresponds to that of an eigenform. All computations are performed in Pari-gp \cite{PARI}.

If $fg \in \E_{k+l}(N,\psi \vp)$, for $fg$ to be an eigenform it is sufficient that the space $\E_{k+l}(N,\psi \vp)$ be one dimensional. Similarly, if $fg \in \S_{k+l}(N,\psi \vp)$, for $fg$ to be an eigenform it is sufficient that $\S_{k+l}(N,\psi \vp)$ be one dimensional. This condition is also necessary for level $1$. All 16 eigenform product identities of level $1$ occur as a result of dimensional considerations. One might think that this condition would be necessary for all levels $N$ as well. Every product identity results in the coefficients of the product being given by the convolution of the coefficients of $f$ and $g$. It seems unreasonable to expect that these coefficients would have the multiplicative properties needed for $fg$ to be an eigenform, unless forced to (see Theorem \ref{thmm3}). 

Nevertheless, we find that it is indeed possible for an eigenform product identity to exist for $\G_1(N)$ when the dimension of the relevant space is greater than one. There are only seven such identities, when the weights of $f$ and $g$ are both greater than 1. They are included in Table \ref{table6}, and correspond to the presence of the function $\Phi_{k,N}$ (see page \pageref{tagg2} for an explanation of this notation).

\section*{Acknowledgments} 
The author would like to thank his adviser Kirti Joshi for all his help and tireless work reading earlier drafts of this paper. Also he would like to thank Dinesh Thakur for his meticulous proof reading and suggestions, Aleksandar Petrov for his helpful comments and advice, and John Brillhart for the mentoring over the years. In addition, the author would like to thank Brad Emmons for sharing a copy of his dissertation. Finally, the author would like to thank his mother.

\section{Preliminaries} \label{secc2}
We review the definitions related to modular forms. For further details see \cite{diamondshurman05} or \cite{diamondim95}.
Given a function $g :\H \rightarrow \C $ and $\g =\m \in \sl $, we define the weight-$k $ operator $[\g ]_k$ by
\[
(g[\g ]_k)(z):=(cz+d)^{-k}g(\g (z)),
\]
where the action of $\g :\H \rightarrow \H $ is defined as usual by
\[
\g (z):=\frac{az+b}{cz+d}.
\]
The two subgroups of $\sl $ of interest to us are
\[
\G_0(N):= \left \{\m \in \sl : \m \equiv \begin{pmatrix}* & *\\ 0 & * \end{pmatrix} \Mod{N} \right\}
\]
and
\[
\G_1(N):= \left \{\m \in \sl : \m \equiv \begin{pmatrix}1 & *\\ 0 & 1 \end{pmatrix} \Mod{N} \right\}.
\]
The \emph{cusps} of a subgroup $\G$ of $\sl $ are the equivalence classes of $\Q \cup \{ \i \}$ under the action of $\G $.
We define a \emph{modular form} of weight $k$ with respect to $\G$ to be a holomorphic function $g :\H \rightarrow \C $ satisfying 
\[
g[\g]_k =g \quad \textrm{for all} \quad \g \in \G ,
\]
provided that $f$ is holomorphic at the cusps of $\G$ (see \cite{diamondshurman05} p.~16 for a precise definition of this last condition). A modular form $g$ with respect to $\G_1(N)$ will have a Fourier expansion of the form
\[
g(z)=\sum_{n=0}^\i a_g(n)q^n, \quad \textrm{where} \quad q=e^{2 \pi i z}.
\]
The space of modular forms of weight $k$ with respect to $\G$ is denoted by $\M_k(\G )$. A cusp form with respect to $\G $ is a modular form that vanishes at every cusp of $\G $. The subspace of cusp forms is denoted by $\S_k(\G )$. The Eisenstein subspace $\E_k(\G )$ is the orthogonal complement of $\S_k(\G )$ in $\M_k(\G )$, with respect to the Petersson inner product. We have the decomposition
\[
\M_k(\G )=\E_k(\G ) \bigoplus \S_k(\G ).
\]

\begin{defi}
A \emph{Dirichlet character} $\chi $ mod $N $ is a homomorphism $\chi :(\Z /N \Z )^\ast \rightarrow \C ^\ast $. We typically view $\chi $ as a function from $\Z \rightarrow \C$. This is done in the natural way by taking $\chi (n)=0$ when $(N ,n )>1$. Given a multiple $M$ of $N$, the character $\chi$ mod $N$ gives rise to a character $\chi^\p$ mod $M$ by
\[
\chi^\p(a)=
\begin{cases}
\chi(a) & \textrm{if} \quad (M,a)=1,\\
0 & \textrm{if} \quad (M,a)>1.
\end{cases}
\]
Note that $\chi^\p$ is determined completely by $\chi$ and the primes dividing $M$ that do not divide $N$. If $M$ consists solely of primes dividing $N$, then $\chi^\p$ and $\chi$ are identical when viewed as functions on $\Z$. We call a character \emph{primitive} if it does not arise from another character of smaller modulus. The \emph{conductor} of $\chi$ is defined to be the modulus of the primitive character that gives rise to $\chi$, and is denoted by $f_\chi$. Thus a character $\chi$ mod $N$ is primitive if and only if $f_\chi=N$. We write $\one{N}$ to denote the trivial character defined modulo $N$. Note that $\one{N}$ is primitive only when $N=1$. A character $\chi $ is \emph{even} if $\chi(-1)=1$ and \emph{odd} if $\chi(-1)=-1$.
\end{defi}
We are frequently only concerned with viewing Dirichlet characters as functions on $\Z $, and will consider two different characters $\chi_1$ mod $N_1$ and $\chi_2$ mod $N_2$ to be equivalent if they give rise to identical functions on $\Z $. This will happen when $\chi_1$ and $\chi_2$ both arise from the same primitive character $\chi$, and any prime $p$ that divides $N_1$ divides $N_2$, and vice versa.

Due to the number of different Dirichlet characters used in this paper, we require a standard way to represent them. Note that the group of Dirichlet characters modulo $N$ is isomorphic to $(\Z/N\Z)^\ast$ (\cite{washington97}, p.~22.), and so will be cyclic if $N$ is a power of an odd prime. For odd $q>3$ a prime or power of a prime, let $1<b<q$ be the smallest generator of the group $(\Z/q\Z)^\ast$. We define $\ch{q,1}$ to be the character modulo $q$ uniquely determined by \label{tagg1}
\[
\ch{q,1}(b)=e^{2 \pi i/\phi(q)},
\]
where $\phi$ is Euler's totient function. For $j=2,\ldots,q-2$, we define $\ch{q,j}:=\ch{q,1}^{\hspace{-3.5mm}j\hspace{2mm}}$. In addition, we define $\ch{3}$, $\ch{4}$ and $\ch{12}$ to be the unique primitive characters defined modulo $3$, $4$ and $12$, respectively. Finally, we define $\ch{8,1}$ to be the character modulo $8$ determined by $\ch{8,1}(5)=\ch{8,1}(7)=-1$, and $\ch{8,2}$ to be the character modulo $8$ determined by $\ch{8,2}(3)=\ch{8,2}(5)=-1$. 

The Hecke operators are a family of linear self-adjoint operators that diagonalize the space $\M_k(\G_1(N))$. For $n$ mod $N$ and $f\in \M_k(\G_1(N))$, the operator $\langle n \rangle$ is defined by
\[
\langle n \rangle f=f[\g]_k,
\]
where $\g=\m$ is any matrix in  $\G_0(N)$ such that $d\equiv n \Mod{N}$. Given a Dirichlet character $\chi$ mod $N$, we define the $\chi$-eigenspace of $\M_k(\G_1(N))$ to be
\[
\M_k(N,\chi):=\{f \in \M_k(\G_1(N)) : \langle n \rangle f=\chi(n)f \quad \forall n \in (\Z/N\Z)^\ast\}.
\]
When $f \in \M_k(N,\chi)$, $f$ is said to have \emph{Nebentypus} $\chi$. We have the orthogonal decomposition according to Nebentypus,
\[
\M_k(\G_1(N))=\bigoplus_\chi \M_k(N,\chi),
\]
where $\chi$ runs through all Dirichlet characters modulo $N$ such that $\chi(-1)=(-1)^k$.
Observe that $\M_k(N,\one{N})=\M_k(\G_0(N))$, and $\M_k(N,\chi)=\{0\}$ if $\chi(-1)\neq (-1)^k$. We define $\E_k(N,\chi)$ and $\S_k(N,\chi)$ analogously to $\M_k(N,\chi)$, and have the orthogonal decompositions
\[
\S_k(\G_1(N))=\bigoplus_\chi \S_k(N,\chi) \quad \textrm{and} \quad \E_k(\G_1(N))=\bigoplus_\chi \E_k(N,\chi).
\]
For $p$ prime, the Hecke operator $T_p$ can be defined by its action on the Fourier expansion of $f(z)=\sum_{n=0}^\i a_f(n)q^n \in \M_k(N,\chi)$. We have
\[
T_pf(z)=\sum_{n=0}^\i (a_f(n p)+\chi(p)p^{k-1}a_f(n/p))q^n.
\]
Note that here and elsewhere, $a_f(n/p)$ is taken to be zero if $p\Nmid n$. Further Hecke operators are defined in terms of $T_p$ and $\langle p \rangle$ by 
\[
T_{p^r}=T_p T_{p^{r-1}}-p^{k-1}\langle p \rangle T_{p^{r-2}} \quad \textrm{for} \quad r\geq 2
\]
and
\[
T_n=\prod_i T_{p_i^{e_i}} \quad \textrm{where} \quad n=\prod_i p_i^{e_i}.
\]
For $f(z) =\sum_{i=0}^\i a_f(i)q^i \in \M_k(N,\chi)$ and arbitrary $n$ we have
\begin{equation}
a_{T_n f}(m)=\sum_{d|(m,n)}\! \chi(d)d^{k-1}a_f(mn/d^2). \label{eqq1}
\end{equation}

\begin{defi}
We say that $f \in \M_k(\G_1(N))$ is an \emph{eigenform} if it is an eigenvector for every Hecke operator $T_n$ and $\langle n \rangle$. Note that we do not require $n$ to be relatively prime to the level $N$.
\end{defi}

\begin{defi} \label{deff1}
We say that a modular form $f(z)=\sum_{n=0}^\i a_f(n)q^n$ is \emph{normalized} if $a_f(1)=1$. Note that all Eisenstein series used will be normalized in this way as well, except where otherwise noted.
\end{defi}
If $f$ is an eigenform, then for every Hecke operator $T_n$ we have from \eqref{eqq1} that
\[
a_{T_n f}(1)=a_f(n)=\l_n a_f(1),
\]
for the eigenvalue $\l_n$. Thus if $f$ is not normalizable, i.e., if $a_f(1)=0$, then $f(z)=a_f(0)$. Therefore every non-constant eigenform is normalizable.

When \label{tagg2} the space $\S_k(\G_0(N))$ is one dimensional, we denote the unique normalized cuspidal eigenform by $\Delta_{k,N}$. When $\S_k(N,\chi)$ is one dimensional, we denote the unique normalized cuspidal eigenform by $\Delta_{k,N,\chi}$. 

Occasionally we will need to reference a specific eigenform in the space $\S_k(\G_0(N))$, when $\S_k(\G_0(N))$ is not one dimensional. In each such instance, there will be only one eigenform in the space $\S_k(\G_0(N))$ that we reference. We write $\Phi_{k,N}$ to denote this specified normalized eigenform, and provide a partial Fourier expansion to identify $\Phi_{k,N}$ uniquely. See page \pageref{tagg4} for a list of the ten $\Phi_{k,N}$ used, along with their corresponding expansions.

We define the Eisenstein series $E_k^{\.\psi,\. \vp}$ (normalized slightly differently than in \cite{diamondshurman05}, p.~129) by
\[
E_k^{\.\psi,\. \vp}(z):=\frac{1}{2}\delta(\psi)L(1-k,\vp)+\sum_{n=1}^\infty \sigma_{k-1}^{\psi,\vp}(n)q^n
\]
where $\delta(\psi)=1$ if $\psi=\one{1}$ and $0$ otherwise, and 
\[
\sigma_{k-1}^{\psi,\vp}(n)=\sum_{d|n}\psi(n/d)\vp(d)d^{k-1}.
\]
For $\psi$ and $\vp$ primitive characters modulo $u$ and $v$, respectively, satisfying $\psi \vp (-1)=(-1)^k$, we have $E_k^{\.\psi,\. \vp} \in \E_k(u v,\psi \vp)$ (except when $k=2$ and $\psi=\vp=\one{1}$). Moreover, for such characters, $E_k^{\.\psi,\. \vp}$ will be an eigenform (see \cite{diamondshurman05}, p.~173). 
Note that $E_k^{\.\psi,\. \vp}$ will be a modular form for certain non-primitive characters $\psi$ and $\vp$ as well. When $k=2$ and $p$ is prime we have the special Eisenstein series defined by
\[
E_2^{\one{1}, \one{1}}(z)-p E_2^{\one{1}, \one{1}}(pz)\in \E_2(p,\one{p})
\]
(see \cite{diamondshurman05} p.~133). Observe that this function is equal to $E_2^{\one{1}, \one{p}}(z)$. We will make use of $E_k^{\.\psi,\. \vp}$ with non-primitive characters $\psi$ and $\vp$ only when $E_k^{\.\psi,\. \vp}$ is a genuine eigenform. For example, $E_2^{\one{2}, \one{2}}(z)=E_2^{\one{1}, \one{2}}(z)-E_2^{\one{1}, \one{2}}(2z)$ is an eigenform in $\E_2(4,\one{4})$.

\begin{theo} [\cite{diamondshurman05}, p.~198] \label{thmm3}
A normalized modular form $f(z)=\sum_{n=0}^\i a_f(n) q^n \in \M_k(N,\chi)$ is an eigenform if and only if the Fourier coefficients $a_f(n)$ satisfy the following two conditions: 
\begin{align}
a_f(p^r)&=a_f(p)a_f(p^{r-1}) - \chi(p)p^{k-1}a_f(p^{r-2}) \quad  p \textrm{ prime}, r \geq 2, \label{eqq2}\\
a_f(mn)&=a_f(m)a_f(n)  \quad \textrm{if }  (m,n)=1. \label{eqq3}
\end{align}
\end{theo}

\begin{theo} [\cite{diamondim95}, p.~115] \label{thmm4}
Let $f(z)=\sum_{n=0}^\i a_f(n)q^n$ be a normalized eigenform. Then the coefficients $a_f(n)$, except for possibly $a_f(0)$, are algebraic integers in some number field.
\end{theo}
We will make use of Sturm's theorem to verify potential eigenform identities.

\begin{theo} [\cite{sturm87}, p.~276]
Let $\mathfrak{m}$ be a prime ideal in the ring of integers $\O$ of some number field. Let $f,g \in \M_k(\G_1(N))$ have Fourier coefficients in $\O$. If 
\[
a_f(n)\equiv a_g(n) \Mod{\mathfrak{m}} \quad \textrm{for all} \quad n < M
\]
then $f \equiv g \Mod{\mathfrak{m}}$, where
\begin{equation}
M=1+\frac{k N^2}{12}\prod_{p\,|N}\! \left(1-\frac{1}{p^2}\right). \label{eqM}
\end{equation}
\end{theo}

\begin{coro} \label{corr1}
Let $f,g \in \M_k(\G_1(N))$ have algebraic integer coefficients. If 
\[
f(z)-g(z)=0+O(q^M),
\]
then $f=g$, where $M$ is given in \eqref{eqM}.
\end{coro}
Thus to verify a potential eigenform identity $f g=h$, it suffices to check that $f(z) g(z) -h(z)=0+O(q^M)$, where $M$ is given in \eqref{eqM}. This is a finite time computation, which we perform in Pari-gp.

\section{Properties of Generalized Bernoulli Numbers} \label{secc3}
The \emph{Bernoulli numbers} $B_n$ are defined by
\[
\frac{t}{e^t-1}=\sum_{n=0}^\i B_n \frac{t^n}{n!}.
\]
The von Staudt-Clausen theorem states that for positive even $k$ the denominator of $B_k$ is the product of all primes $p$ such that $(p-1)|k$. Note that for odd $k>1$ we have $B_k=0$, and that $B_1=-1/2$. 

\begin{defi} [\cite{diamondim95}, p.~44]
Given a Dirichlet character $\chi$ mod $N$, not necessarily primitive, the \emph{generalized Bernoulli numbers} $B_{n,\chi}$ are defined by 
\[
\sum_{a=1}^{N} \chi(a)\frac{t e^{at}}{e^{N t}-1}=\sum_{n=0}^\i B_{n,\chi} \frac{t^n}{n!}.
\]
Note that if $\chi =\one{1}$, then $B_{k,\chi}=B_k$ (except when $k=1$). 
\end{defi}
The generalized Bernoulli numbers $B_{k,\chi}$ arise in conjunction with modular forms due to the fact that 
\[
L(1-k,\chi)=-B_{k,\chi}/k \quad \textrm{for} \quad k \geq 1.
\]
Hence we have
\begin{equation}
E_k^{\.\psi,\. \vp}(z)=-\delta(\psi)\frac{B_{k,\vp}}{2k}+\sum_{n=1}^\infty \sigma_{k-1}^{\psi,\vp}(n)q^n. \label{eqq4}
\end{equation}
For primitive $\chi$ mod $N$ we have
\begin{equation}
B_{k,\chi}=\frac{1}{f_{\chi}}\sum_{a=1}^{f_{\chi}} \sum_{i=0}^k \binom{k}{i} \chi(a) B_i f_{\chi}^i a^{k-i}. \label{eqq5}
\end{equation}
For details see \cite{washington97} p.~32. Recall that $\chi$ mod $N$ is primitive if and only if $f_\chi=N$. For non-primitive $\chi$ mod $N$, the value of $B_{k,\chi}$ is given by
\begin{equation}
B_{k,\chi}=B_{k,\ch{0}} \!\prod_{p \, |N}\! \big(1-\ch{0}(p)p^{k-1}\big), \label{eqq6}
\end{equation}
where $\ch{0}$ is the primitive character associated to $\chi$ (see \cite{washington97}, p.~206). If $\chi$ is even and $k$ is odd, or vice versa, then $B_{k,\chi}=0$ (except when $k=1$ and $\chi=\one{N}$). Note that the converse is true for $k>1$. From \eqref{eqq6} we see that $|B_{k,\chi}|\geq |B_{k,\ch{0}}|$ when $k \geq 2$.

\begin{prop} \label{propp1}
Let $f(z) =\sum_{n=0}^\i a_f(n)q^n\in \M_k(N,\chi)$ be a normalized eigenform with $a_f(0)\neq 0$, where $\chi$ mod $N$ is not necessarily primitive. Then $f=E_k^{\one{1},\. \chi}$.
\end{prop}
\begin{proof}
If an eigenform $f\in \M_k(N,\chi)$ with $a_f(0) \neq 0$ is normalized so that $a_f(1)=1$, then
\[
a_f(n)=\sum_{d|n} \chi(d) d^{k-1}\quad \textrm{for} \quad n>0
\]
(see \cite{diamondim95}, p.~54). This implies that there is \emph{at most} one normalized eigenform with non-zero constant term in $\M_k(N,\chi)$. If $\chi$ is primitive then $E_k^{\one{1},\. \chi}$ is a modular form (and an eigenform), and thus $f=E_k^{\one{1},\. \chi}$. So we consider a non-primitive $\chi$ mod $N$. Let $\ch{0}$ be the primitive character associated to $\chi$ and let $N_1$ be the product of all primes dividing $N$ but not $f_\chi$. Consider the eigenform%
\footnote{If $k=2$ and $\chi=\one{N}$ then $g=E_2^{\one{1}, \one{1}}$ is not a modular form, but the linear combination given in \eqref{eqq7} nevertheless is. To see this let $p$ be a prime dividing $N_1$. Observe that
\[
\sum_{d|N_1}\mu(d)d^{k-1}E_2^{\one{1}, \one{1}}(dz)=\sum_{d|\frac{N_1}{p}}\mu(d)d^{k-1}E_2^{\one{1}, \one{p}}(dz),
\]
and note that $E_2^{\one{1}, \one{p}}(z)=E_2^{\one{1}, \one{1}}(z)-p E_2^{\one{1}, \one{1}}(pz)$ is a modular form.}
$g=E_k^{\one{1},\. \ch{0}} \in \E_k(f_\chi,\ch{0})$ (hence $a_g(n)=\sum_{d|n} \ch{0}(d) d^{k-1}$ for $n>0$). Now consider the modular form
\begin{equation}
h(z)=\sum_{d|N_1} \mu(d)\ch{0}(d)d^{k-1}g(dz) \label{eqq7}
\end{equation}
where $\mu$ is the M\"obius function. Observe that $h \in \E_k(N_1 f_\chi, \chi^\p)$, where $\chi^\p$ is $\chi$ viewed modulo $f_\chi N_1$ (note that $\chi^\p$ and $\chi$ are identical as functions on $\Z$). We have
\[
a_h(n)=\sum_{d|N_1} \mu(d)\ch{0}(d) d^{k-1}a_g(n/d).
\]
From this we see that if $(n,N_1)=1$ then $a_h(n)=a_g(n)$. We will show that $h$ is an eigenform by showing the coefficients $a_h(n)$ satisfy the conditions of Theorem \ref{thmm3}. We take $n>0$ and write $n=m M$, where $m$ is the largest divisor of $n$ relatively prime to $N_1$. Thus $(n,N_1)=1$, and $p|M \Rightarrow p|N_1$. We first show that $a_h(m M)=a_g(m)$. We have 
\begin{align*}
a_h(m M)&=\sum_{d|N_1} \mu(d)\ch{0}(d) d^{k-1}a_g\!\left(\frac{m M}{d}\right)\\
&=a_g(m) \sum_{d|M} \mu(d)\ch{0}(d) d^{k-1}a_g\!\left(\frac{ M}{d}\right)
\end{align*}
since $(m,d)=1$ when $d|N_1$ (here we use the multiplicative properties of $a_g(n)$ given in Theorem \ref{thmm3}). We want to show that the arithmetic function
\[
\rho(M):=\sum_{d|M} \mu(d)\ch{0}(d) d^{k-1}a_g(M/d)
\]
is identically $\one{}$. The function $\nu(n)=\mu(n)\ch{0}(n) n^{k-1}$ is multiplicative, as well as $a_g(n)$. Observe that $\rho$ is the Dirichlet convolution of $\nu$ and $a_g$, and is therefore also multiplicative. Thus to show $\rho=\one{}$, we need only observe that $\rho(p^l)=1$ for any $p$ prime and $l\geq 1$. We have
\begin{align*}
\rho(p^l)&=a_g(p^l)-\ch{0}(p)p^{k-1}a_g(p^{l-1})\\
&= \sum_{i=0}^l \ch{0}(p^i)p^{i(k-1)} - \ch{0}(p)p^{k-1} \sum_{i=0}^{l-1} \ch{0}(p^i)p^{i(k-1)}=1.
\end{align*}
Therefore we have $a_h(m M) = a_g(m)$. From this we see that the coefficients $a_h(n)$ satisfy both conditions of Theorem \ref{thmm3}. Thus $h$ is an eigenform. Since $h\in \M_k(f_\chi N_1,\chi^\p)$ is an eigenform with non-zero constant term, we have $a_h(n)=\sum_{d|n}\chi^\p(d)d^{k-1}=\sum_{d|n}\chi(d)d^{k-1}$, when $n>0$. The value of $a_h(0)$ is given by 
\[
-\frac{B_{k,\ch{0}}}{2k}\sum_{d|N_1} \mu(d)\ch{0}(d)d^{k-1}=-\frac{B_{k,\ch{0}}}{2k}\prod_{p\,|N_1} \!  \big(1-\ch{0}(p)p^{k-1}\big)=-\frac{B_{k,\chi}}{2k}.
\]
Therefore $h=E_k^{\one{1},\. \chi}$.
\end{proof}

\begin{prop}
Consider a non-trivial character $\chi$ mod $N$ of order $m$ \lp{}i.e., $\chi^m=\one{N}$\rp{}. If we let $K$ be the number field generated by the values of $\chi(a)$, then $K = \Q[\z{m}]$, where $\z{m}=e^{2 \pi i/m}$. Therefore $B_{k,\chi} \in \Q[\z{m}]$. More specifically, we have that $f_\chi B_{k,\chi}/k$ is an algebraic integer \lp{}see \cite{carlitz59}, p.~176\rp{}, i.e.,
\begin{equation}
f_\chi  B_{k,\chi}/k \in \Z[\z{m}]. \label{eqq8}
\end{equation}
For even $\chi$ with $f_\chi$ odd we have \lp{}see \cite{carlitz59}, p.~182\rp{}
\begin{equation}
B_{k,\chi}/k \equiv 0 \Mod{2}. \label{eqq9}
\end{equation}
\end{prop}

\begin{prop} \label{lemm1}
Let $k \geq 2$ and $\chi$ mod $f_\chi$ be a primitive Dirichlet character such that $\chi(-1)=(-1)^k$. We have 
\[
2\frac{\zeta(2k)}{\zeta(k)} k!(2 \pi)^{-k} f_{\chi}^{k-1/2} \leq \left|B_{k,\chi}\right| \leq 2\zeta(k) k!(2 \pi)^{-k} f_{\chi}^{k-1/2}.
\]
\end{prop}
\begin{proof}
The functional equation of $L(s,\chi)$ can be written as
\begin{equation}
L(1-s,\chi)=\frac{\G(s)2i^\delta}{\tau(\overline{\chi})}\cos\!\left(\frac{\pi(s-\delta)}{2}\right)\!\left(\frac{f_\chi}{2\pi}\right)^{\!s}\!L(s,\overline{\chi}), \label{eqq10}
\end{equation}
(see \cite{washington97}, p.~30) where $\delta=0$ if $\chi(-1)=1$, $\delta=1$ if $\chi(-1)=-1$, and $\tau(\chi)=\sum_{a=1}^{f_\chi}\chi(a)e^{2 \pi i a/f_{\chi}}$ is the Gauss sum of $\chi$. We make use of the fact that 
\[
|\tau(\chi)|=\sqrt{f_\chi},
\]
\[
L(s,\chi)=\prod_p(1-\chi(p)p^{-s})^{-1} \quad \textrm{for} \quad \re(s)>1
\]
and
\[
L(1-k,\chi)=-B_{k,\chi}/k \quad \textrm{for} \quad k \geq 1.
\]
We set $s=k$ in \eqref{eqq10} and take the absolute value to obtain
\[
|B_{k,\chi}|=2k!(2 \pi)^{-k} f_{\chi}^{k-1/2} \prod_p  |1-\overline{\chi}(p)p^{-k}|^{-1}.
\]
We bound the term $|1-\overline{\chi}(p)p^{-k}|^{-1}$ using
\[
(1+p^{-k})^{-1} \leq |1-\overline{\chi}(p)p^{-k}|^{-1} \leq (1-p^{-k})^{-1},
\]
and observe that 
\[
(1+p^{-k})^{-1}=\frac{(1-p^{-2k})^{-1}}{(1-p^{-k})^{-1}}.
\]
We use the Euler product formula for the Riemann zeta function,
\[
\zeta(s)=\prod_p(1-p^{-s})^{-1} \quad \textrm{for} \quad \re(s)>1,
\]
to obtain lower and upper bounds for the entire product. We have
\[
\zeta(2k)/\zeta(k) \leq \prod_p |1-\overline{\chi}(p)p^{-k}|^{-1} \leq \zeta(k).
\]
The inequalities follow directly from this.
\end{proof}

\begin{prop} \label{propp2}
There are only 41 pairs $(k,\chi)$ with $k \geq 2$ such that $2k/B_{k,\chi}$ is an algebraic integer \lp{}$\chi$ not necessarily primitive\rp{}. A complete list along with the corresponding integral values of $2k/B_{k,\chi}$ is given in Tables \ref{table1}, \ref{table2} and \ref{table3}. 
\end{prop}

\begin{rmk} \label{remm1}
Note that here we are viewing $\chi$ mod $N$ as a function on $\Z$, rather than a function on $(\Z/N\Z)^\ast$. So if the characters $\chi$ and $\chi^\p$ define the same function on $\Z$, then we consider the pairs $(k,\chi)$ and $(k,\chi^\p)$ to be the same. For example, $\ch{3}$ and $\ch{9,3}$ are the same function on $\Z$, but have different moduli. The pairs $(k,\ch{3})$ and $(k,\ch{9,3})$ are fundamentally the same, and give rise to the same value of $2k/B_{k,\chi}$.
\end{rmk}
\begin{proof}
We first consider the case where $\chi=\one{1}$. Recall that for $k>1$ we have $B_{k,\one{1}}=B_k$. It is known that $2k/B_k$ is an integer only when $k\in \{1,2,4,6,8,10,14\}$. This can be seen from the upper bound on $2k/B_{k,\one{1}}$ implied by Proposition \ref{lemm1}. The six $k\geq 2$ such that $2k/B_{k,\one{1}}$ is integral yield the values contained in Table \ref{table1}.
\begin{table} [ht]
\caption{} \label{table1}
\begin{tabular}{|c|c|c|}
\hline
\ra $\frac{4}{B_{2,\one{1}}}=24$ & $\frac{8}{B_{4,\one{1}}}=-240$ & $\frac{12}{B_{6,\one{1}}}=504$\\
\hline
\ra $\frac{16}{B_{8,\one{1}}}=-480$ & $\frac{20}{B_{10,\one{1}}}=264$ & $\frac{28}{B_{14,\one{1}}}=24$\\
\hline
\end{tabular}
\end{table}

We now consider the case where $\chi=\one{N}$. From $\eqref{eqq6}$ we have
\[
\frac{2k}{B_{k,\one{N}}} =\frac{2k}{B_{k,\one{1}}} \prod_{p\,|N} \! \left(1-\one{1}(p)p^{k-1}\right)^{-1}.
\]
If $2k/B_{k,\one{N}}$ is integral then $2k/B_{k,\one{1}}$ must also be integral, and so we have $k\in\{2,4,6,8,10,14\}$. The function $\one{N}$ depends only on the primes dividing $N$, so without loss of generality we take $N$ to be square-free (see Remark \ref{remm1}). An inspection shows that if $N>1$ and $k\in \{4,6,8,10,14\}$, $2k/B_{k,\one{N}}$ will not be an integer. To see this observe that if $p\geq 11$, then $|1-p^{k-1}|>504$ for $k\geq 4$. Checking the finite number of pairs $(k,\one{N})$ not covered we see that $2k/B_{k,\one{N}}$ is never integral. So we restrict ourselves to the case $k=2$. We have
\[
\frac{4}{B_{2,\one{N}}}=\frac{24}{\prod_{p|N}(1-p)}.
\]
There are a total of 17 (square-free) values of $N>1$ such that $4/B_{2,\one{N}}$ is an integer. The values are given in Table \ref{table2}.
\begin{table} [ht]
\caption{} \label{table2}
\begin{tabular}{|c|c|c|c|c|}
\hline
\ra  $\frac{4}{B_{2,\one{2}} }=-24$ & $\frac{4}{B_{2,\one{3}}}=-12$ & $\frac{4}{B_{2,\one{5}}}=-6$ & $\frac{4}{B_{2,\one{6}}}=12$ & $\frac{4}{B_{2,\one{7}}}=-4$ \\
\hline
\ra $\frac{4}{B_{2,\one{10}}}=6$ & $\frac{4}{B_{2,\one{13}}}=-2$ & $\frac{4}{B_{2,\one{14}}}=4$ & $\frac{4}{B_{2,\one{15}}}=3$ & $\frac{4}{B_{2,\one{21}}}=2$ \\
\hline
\ra $\frac{4}{B_{2,\one{26}}}=2$ & $\frac{4}{B_{2,\one{30}}}=-3$ & $\frac{4}{B_{2,\one{35}}}=1$ & $\frac{4}{B_{2,\one{39}}}=1$ & $\frac{4}{B_{2,\one{42}}}=-2$ \\
\hline
\ra $\frac{4}{B_{2,\one{70}}}=-1$ & $\frac{4}{B_{2,\one{78}}}=-1$ & & & \\
\hline
\end{tabular}
\end{table}

We now consider the case where $\chi$ mod $f_\chi$ is a primitive non-trivial character. For any $\sigma \in \Gal(\Qbar/\Q)$, we define $\chi^\sigma$ to be the character modulo $f_\chi$ determined by
\[
\chi^\sigma(a)=\sigma\left(\chi(a)\right).
\]
Note that $\chi^\sigma$ will be primitive if $\chi$ is primitive. From the representation of $B_{k,\chi}$ given in \eqref{eqq5} we see that
\[
\sigma\left({B_{k,\chi}}\right)={B_{k,\chi^\sigma}}.
\]
If $2k/B_{k,\chi}$ is an algebraic integer, then so is $\sigma(2k/B_{k,\chi})=2k/B_{k,\chi^\sigma}\!$ for every $\sigma \in \Gal(\Qbar/\Q)$.
Given an algebraic integer $\b$, if the absolute value of every Galois conjugate of $\b$ is less than $1$, then $\mathcal{N}(\b)< 1$, and thus $\b=0$ (here $\mathcal{N}$ denotes the absolute norm). Therefore, if $2k/B_{k,\chi}$ is integral, then there must exist a $\sigma \in \Gal(\Qbar/\Q)$ such that $|\sigma(2k/B_{k,\chi})|=|2k/B_{k,\chi^\sigma}|\! \geq 1$.

For positive integers $k$ and $f_\chi$, we define
\[
C(k ,f_\chi ):=\frac{\zeta(2k)(k-1)!f_\chi^{k-1/2}}{\zeta (k)(2\pi)^k}. \label{tagg3}
\]
From Proposition $\ref{lemm1}$, if $\chi$ is primitive and of the same parity as $k>1$, we have  
\[
C(k , f_\chi ) \leq \left| \frac{B_{k , \chi}}{2k}\right|.
\]
Note that for fixed $k$, $C(k,f_\chi)$ is strictly increasing as a function of $f_\chi $. We are interested in finding pairs $(k ,f_\chi )$ such that $C(k ,f_\chi )>1$. Given such a pair, we see that $2k/B_{k,\chi}$ will not be integral for primitive $\chi$. We have $C(k,f_\chi) >1$ when $f_\chi \geq 3$ and $k \geq 7$. Note that if $\chi$ is non-trivial, then $f_\chi \geq 3$. Thus if $k \geq 7$ and $\chi$ is non-trivial, $2k/B_{k,\chi}$ will not be an algebraic integer. 

We also have $C(k,f_\chi) >1$ when $f_\chi \geq 16$ and $2 \leq k \leq 6$. Thus there is a finite list of weights $k$ ($2 \leq k\leq 6$) and characters $\chi$ mod $f_\chi$ ($f_\chi \leq 16)$ to check where $2k/B_{k,\chi}$ could potentially be an algebraic integer. We use Pari-gp to perform the computations of the values of $2k/B_{k,\chi}$. We find a total of 18 integral values, given in Table \ref{table3}. Note that $\z{6}=e^{\pi i/3}$.
\begin{table} [ht]
\caption{} \label{table3}
\begin{tabular}{|c|c|c|}
\hline
\ra $\frac{4}{B_{2,\ch{5,2}}}=5$ & $\frac{4}{B_{2,\ch{7,2}}}=\z{6}+2$ & $\frac{4}{B_{2,\ch{7,4}}}=-\z{6}+3$\\
\hline
\ra $\frac{4}{B_{2,\ch{8,2}}}=2$ & $\frac{4}{B_{2,\ch{9,2}}}=-\z{6}+2$ & $\frac{4}{B_{2,\ch{9,4}}}=\z{6}+1$\\
\hline
\ra $\frac{4}{B_{2,\ch{12}}}=1$ & $\frac{4}{B_{2,\ch{13,4}}}= -\z{6}^{\!\!2}$ & $\frac{4}{B_{2,\ch{13,8}}}=\z{6}^{\!\!2}+1$\\
\hline
\ra $\frac{4}{B_{2,\ch{13,6}}}=1$ & $\frac{6}{B_{3,\ch{3}}}=9$ & $\frac{6}{B_{3,\ch{4}}}=4$\\
\hline
\ra $\frac{6}{B_{3,\ch{5,1}}}=2-i$ & $\frac{6}{B_{3,\ch{5,3}}}=2+i$ & $\frac{6}{B_{3,\ch{7,1}}}=-\z{6}+1$\\
\hline
\ra $\frac{6}{B_{3,\ch{7,5}}}=\z{6}$ & $\frac{8}{B_{4,\ch{5,2}}}=-1$ & $\frac{10}{B_{5,\ch{3}}}=-3$\\
\hline
\end{tabular}
\end{table}

Finally, we consider the case where $\chi$ mod $N$ is a non-primitive, non-trivial character. Let $\ch{0}$ be the primitive character associated to $\chi$. If $2k/B_{k,\chi}$ is integral, then $2k/B_{k,\ch{0}}$ must also be integral. Therefore we can restrict our search to the 18 pairs $(k,\ch{0})$ given in Table \ref{table3}. Let $\nu=2k/B_{k,\ch{0}}$ be one such integral value. We want to eliminate the possibility that
\[
\nu_p = \frac{\nu}{(1-\ch{0}(p)p^{k-1})}
\]
is an algebraic integer, for a prime $p$ dividing $N$ but not $f_{\chi}$. Note that the largest possible value of $|\nu|$ is $9$. Therefore, if $p \geq 11$ then $|\nu_p|<1$ (and likewise for all conjugates of $\nu_p$). 
Thus  $\nu/(1-\ch{0}(p)p^{k-1})$ cannot be an algebraic integer. All that remains is to check the value of $\nu_p$ when $p\in \{2,3,5,7\}$, for each of the 18 pairs $(k,\ch{0})$. An inspection yields no additional pairs $(k,\chi)$ with $2k/B_{k,\chi}$ integral.
\end{proof}

\section{Proof Preliminaries} \label{secc4}
The proof of Theorem \ref{thmm1} naturally splits into two cases. We need not consider when both $f$ and $g$ have Fourier expansions starting with $q$, for in this case the product $fg$ will have a Fourier expansion starting with $q^2$, and therefore cannot be an eigenform. So without loss of generality we assume that $f$ has an expansion with a non-zero constant term. We consider the case where the constant term in the expansion of $g$ is zero and the case where the constant term of $g$ is non-zero.

\begin{case} \label{case1} Let
\begin{align*}
f(z)&=\sum_{n=0}^\i a_f(n)q^n \in \M_k(N,\psi), \hspace{-5ex} & a_f(0)& \neq 0 \quad \textrm{and} \\ 
g(z)&=\sum_{n=0}^\i a_g(n)q^n \in \M_l(N,\vp),  \hspace{-5ex} & a_g(0)& = 0.
\end{align*}
\end{case}

\begin{case} \label{case2} Let
\begin{align*}
f(z)&=\sum_{n=0}^\i a_f(n)q^n \in \M_k(N,\psi), \hspace{-5ex} & a_f(0)& \neq 0 \quad \textrm{and} \\ 
g(z)&=\sum_{n=0}^\i a_g(n)q^n \in \M_l(N,\vp),  \hspace{-5ex} & a_g(0)& \neq 0.
\end{align*}
\end{case}

\begin{theo} \label{thmm5}
There are only 55 tuples $(N,k,l,\psi,\vp)$, $k,l>1$, such that there exist $f$ and $g$ eigenforms conforming to Case \ref{case1} with $fg$ also an eigenform. There are a total of 55 such eigenform identities, with one per tuple.
\end{theo}

\begin{theo} \label{thmm6}
There are only seven tuples $(N,k,l,\psi,\vp)$, $l\geq k>1$, such that there exist $f$ and $g$ eigenforms conforming to Case \ref{case2} with $fg$ also an eigenform. There are a total of six such eigenform identities, as the identity \eqref{eqq29} corresponds to two separate tuples.
\end{theo}
These 61 eigenform product identities account for all eigenform identities $h=fg$ with $f$, $g$ and $h$ eigenforms with respect to $\M(\G_1(N))$ of weight greater than $1$.

\section{Proof of Theorem \ref{thmm5}} \label{secc5}
We assume that $f(z)=\sum_{n=0}^\i a_f(n)q^n \in \M_k(N,\psi)$, $a_f(0)\neq 0$ and $g(z) =\sum_{n=0}^\i c_n q^n\in \M_l(N,\vp)$, $c_0= 0$ are eigenforms with $k,l>1$, and that the product $fg \in \M_{k+l}(N,\psi\vp)$ is also an eigenform. Both $\psi$ and $\vp$ are defined modulo $N$, not necessarily primitive. Without loss of generality, we normalize $g$ so that $c_1=1$ (as usual) and $f$ differently, so that $a_f(0)=1$. From Proposition \ref{propp1} we have
\[
f(z)=-\frac{2k}{B_{k,\psi}}E_k^{\one{1},\. \psi}(z)=1-\frac{2k}{B_{k,\psi}}\sum_{n=1}^\i \sigma_{k-1}^{\one{1},\psi}(n)q^n.
\] 
Throughout Section \ref{secc5} we use
\[
\o=-\frac{2k}{B_{k,\psi}}, \quad \A_p=p^{k-1}\psi(p) \quad \textrm{and} \quad \B_p=p^{l-1}\vp(p)
\]
for $p$ prime. Using $\o$ and $\A_p$'s, the expansion of $f$ is given by
\begin{align*}
f(z) &= 1+\o \big[q+(1+\A_2)q^2+(1+\A_3)q^3+(1+\A_2+\A_2^2)q^4\\
& \qquad +(1+\A_5)q^5+(1+\A_2)(1+\A_3)q^6+(1+\A_7)q^7+\cdots \big].
\end{align*}
We write the normalized eigenform $g(z) =\sum_{n=1}^\i c_n q^n$ as 
\begin{align*}
g(z) &= q+c_2q^2+c_3q^3+(c_2^2- \B_2)q^4+c_5q^5+c_2c_3q^6+c_7q^7\\
 & \qquad +(c_2^3-2 \B_2c_2)q^8+(c_3^2- \B_3)q^9+c_2c_5q^{10}+c_{11}q^{11}+\cdots,
\end{align*}
using the multiplicative properties of the coefficients $c_n$ given in Theorem \ref{thmm3}. 
The product $fg$ has the expansion
\begin{align*}
f(z)g(z) &= q+(\o+c_2)q^2+(\o(\A_2+c_2+1) + c_3)q^3\\
& \qquad +(\o(\A_2 c_2 + \A_3 + c_2 + c_3 +1) -  \B_2 + c_2^2 )q^4+ \cdots .
\end{align*}

\begin{prop} \label{propp3}
Given an eigenform $g(z)=\sum_{n=1}^\i c_n q^n$, 
if $g E_k^{\one{1},\. \psi}$ is an eigenform with $k>1$, then $(k,\psi)$ is one of the 40 pairs in Table \ref{table22}. 
\end{prop}

\begin{table} [ht]
\caption{} \label{table22}
\begin{tabular}{|c|c|c|c|c|c|c|c|}
\hline
$(2,\one{2})$ & $(2,\one{3})$ & $(2,\one{5})$ & $(2,\one{6})$ & $(2,\one{7})$ & $(2,\one{10})$ & $(2,\one{13})$ & $(2,\one{14})$ \\
\hline
$(2,\one{15})$ & $(2,\one{21})$ & $(2,\one{26})$ & $(2,\one{30})$ & $(2,\one{35})$ & $(2,\one{39})$ & $(2,\one{42})$ & $(2,\one{70})$ \\
\hline
$(2,\one{78})$ & $(2,\ch{5,2})$ & $(2,\ch{7,2})$ & $(2,\ch{7,4})$ & $(2,\ch{8,2})$ & $(2,\ch{9,2})$ & $(2,\ch{9,4})$ & $(2,\ch{12})$ \\
\hline
$(2,\ch{13,4})$ & $(2,\ch{13,6})$ & $(2,\ch{13,8})$ & $(3,\ch{3})$ & $(3,\ch{4})$ & $(3,\ch{5,1})$ & $(3,\ch{5,3})$ & $(3,\ch{7,1})$\\
\hline
$(3,\ch{7,5})$ & $(4,\one{1})$ & $(4,\ch{5,2})$ & $(5,\ch{3})$ & $(6,\one{1})$ & $(8,\one{1})$ & $(10,\one{1})$ & $(14,\one{1})$\\
\hline
\end{tabular}
\end{table}

\begin{proof}
From Theorem \ref{thmm4} we see that the coefficients $c_2$ and $\o+c_2$ of $q^2$ in the Fourier expansions of $g$ and $fg$, respectively, are algebraic integers. Therefore $\o=-2k/B_{k,\psi}$ must be an algebraic integer. From Proposition \ref{propp2} there are 41 pairs $(k,\psi)$ such that $-2k/B_{k,\psi}$ is integral. We exclude the pair $(2,\one{1})$, as $E_2^{\one{1},\one{1}}$ is not a modular form.
\end{proof}

This is a key result that allows us to restrict our search for eigenform identities to a finite number of eigenforms $f$, thus making explicit effective calculations possible. 

Our goal is to determine the expansion of $g$, given $f=E_k^{\one{1},\. \psi}$ (having fixed the pair $(k,\psi)$). Using Theorem \ref{thmm3}, $c_n$ is determined in terms of $c_p$ and $\B_p$, for primes $p$ dividing $n$. Note that $\B_p$ is given by 
\[
\B_p=(c_p)^2-c_{p^2}.
\]
 
To determine the values of $c_p$ and $\B_p$, it suffices to determine $c_2$ and $ \B_2$. We have 
\[
\left(1+\sum_{n=1}^\i a_f(n)q^n\right) \left(\sum_{n=1}^\i c_n q^n\right) = \sum_{n=1}^\i d_n q^n.
\]
Therefore we have 
\[
d_{n+1}=c_{n+1} + \sum_{i=1}^{n}c_i a_f(n+1-i).
\]
Solving for $c_n$ we obtain 
\begin{equation}
c_n=\frac{1}{a_f(1)}\left(d_{n+1}-c_{n+1}-\sum_{i=1}^{n-1}c_i a_f(n+1-i) \right). \label{eqq15}
\end{equation} 
Let $p>2$ be prime, so that $p+1$ is composite. Observe that equation \eqref{eqq15} determines $c_p$ uniquely in terms of $c_n$'s with $n<p$, since we can express $c_{p+1}$ and $d_{p+1}$ in terms of previously determined $c_j$ and $d_j$ using their multiplicative properties. Note the one exception to this: if $p=3$ then $c_3$ is given in terms of $c_4=(c_2)^2-\B_2$, and hence we must know the value of either $\B_2$ or $c_4$ to determine $c_3$. Using \eqref{eqq15} we can similarly determine $c_{p^2}$, and thus determine $\B_p$.

Therefore, given knowledge of $c_2$ and $\B_2$, we can determine the expansion of $g$ to any desired precision, and then check if that expansion corresponds to a known eigenform. If $g$ is an eigenform, we then identify the expansion of $fg$ with a known eigenform $h$, and verify the identity $h=fg$ using Corollary \ref{corr1}.

We now must determine the values of $c_2$ and $\B_2$. To this end, we derive relations from the coefficients of the product 
\[
f(z)g(z)=\sum_{n=1}^\i d_n q^n,
\]
using the multiplicative characteristics set forth in Theorem \ref{thmm3}. For every composite index we obtain a relation from one of the following,
\begin{align}
d_{p^r}&=d_{p^{r-1}}d_r-\psi(p)\vp(p)p^{k+l-1}d_{p^{r-2}}=d_{p^{r-1}}d_r-p\A_p \B_p d_{p^{r-2}}, \label{eqq12} \\
d_{mn}&=d_m d_n \qquad \textrm{ for }  (m,n)=1. \label{eqq13}
\end{align}
The following are a few examples of the relations generated from \eqref{eqq12} and \eqref{eqq13}. From $d_4=d_2^2-\psi(2) \vp(2) 2^{k+l-1}=d_2^2-2\A_2 \B_2$ we have
\begin{equation}
\o^2+\o(-\A_2 c_2- \A_3+c_2-c_3-1)-2 \A_2 \B_2+\B_2=0. \label{eqq14}
\end{equation}
From $d_6=d_2d_3$ we have
\[
\o( \A_2 + c_2 + 1) +  \A_2( - \A_2c_2+\B_2  -  c_2^2) - \A_3c_3 - \A_5+  \B_2  -c_5-1=0.
\]
From $d_8=d_4d_2-\psi(2) \vp(2) 2^{k+l-1}d_2=d_4d_2-2\A_2 \B_2 d_2$ (and subtracting a multiple of \eqref{eqq14}) we obtain
\begin{align*}
& \o(\A_2c_2 +  \A_3 + c_3 + 1) +  \A_2(\A_2 \B_2 -\A_2c_2^2- \A_3c_2-  \B_2  + c_2^2  - c_2c_3-c_2) \\
& +  \A_3c_2  -\A_3 c_5 -\A_5c_3 - \A_7+ c_2c_3 + c_2 -c_3 -c_5 -c_7 - 1=0.
\end{align*}

To determine the values of $c_2$ and $\B_2$ it is convenient to consider the cases $2|N$ and $2 \Nmid N$ separately. There are 12 pairs $(k,\psi)$ in Table \ref{table22} such that the modulus of $\psi$ is even, and 28 pairs such that the modulus of $\psi$ is odd.

\begin{caseA} Assume that $2|N$. Fix a pair $(k,\psi)$ from Table \ref{table22}. We have $ \A_2= \B_2=0$. We combine equations generated by \eqref{eqq13} and \eqref{eqq12} to eliminate $c_3$, $c_5$, $c_7$ and $ \B_3$, resulting in an equation determining $c_2$ in terms of $\o$, $ \A_3$, $ \A_5$ and $ \A_7$. The key point is that these values are known to us since we have fixed $(k,\psi)$. We have
\begin{equation}
\o c_2^3+j_2c_2^2+j_1c_2+j_0=0 \label{eqq16}
\end{equation}
where
\begin{align*}
j_2&=-\o^2 + \o(-2 \A_3 +  \A_5 + 2) + 3 \A_3^2 + 5 \A_3 - 2,\\
j_1&=-\o^2 + \o(-3 \A_3^3 - 2 \A_3^2 + 3 \A_3( \A_5 + 5) - 2 \A_5 +  \A_7 - 1)\\
& \qquad + 3 \A_3^4 + 5 \A_3^3 - 6 \A_3^2( \A_5 - 8) +  \A_3(- \A_5 + 3 \A_7 - 4) +  \A_5 -  \A_7 + 2,\\
j_0&=\o^3 + \o^2(4 \A_3^2 + 2 \A_3 -  \A_5 - 4)\\
& \qquad +\o(3 \A_3^4 +  \A_3^3 - 6 \A_3^2( \A_5 - 12) +  \A_3(4 \A_5 + 3 \A_7 - 4) +  \A_5 - 2 \A_7 + 5) \\
& \qquad \qquad -3 \A_3^5 - 8 \A_3^4 +  \A_3^3(9 \A_5 + 6) +  \A_3^2(12 \A_5 - 6 \A_7 + 11) \\
& \qquad \qquad \qquad + \A_3(-3 \A_5^2 - 2 \A_5 - 4 \A_7 + 2) +  \A_5^2 -  \A_5 + 2 \A_7 - 2.
\end{align*}
We compute the roots of \eqref{eqq16}, and reject any values of $c_2$ that are not algebraic integers. Since the degree of \eqref{eqq16} is three, this computation is straightforward.
\end{caseA}

\begin{caseB} Assume that $2 \Nmid N$. Again fix a pair $(k,\psi)$ from Table \ref{table22}. Because $ \A_2, \B_2\neq 0$, we cannot (readily) determine $c_2$ in terms of $\o$ and $ \A_p$'s. Instead we use a similar technique to determine an equation for $ \B_2$ in terms of $\o$ and $ \A_p$'s. A valid root must have absolute value equal to a power of $2$, so an approximation of the roots suffices. Given $ \B_2$, we can then determine $c_2$. 
The number of terms involved makes determining a general explicit equation for $ \B_2$ in terms of arbitrary $\o$ and $\A_p$'s, comparable to $\eqref{eqq16}$, infeasible. We \emph{can}, however, determine the desired equation when we specify the values of $\o$ and $\A_p$ (since we have fixed the pair $(k,\psi)$). The computations performed in Pari-gp make the tedious analysis manageable.
\end{caseB}
A sample computation for the pair $(4,\ch{5,2})$ proceeds as follows. The expansion of $f=E_4^{\one{1},\. \ch{5,2}}$ is given by
\[
f(z)=1 + q - 7q^2 - 26q^3 + 57q^4 + q^5 + 182q^6 - 342q^7 - 455q^8 + \ldots .
\]
We combine equations generated by \eqref{eqq13} and \eqref{eqq12} to determine $\B_2$ in terms of $\o$ ($\o=1$) and $\A_p$'s ($\A_p=p^3 \ch{5,2}(p)$). We obtain
\begin{align*}
& \hspace{-.2in} 99087935063204566214333650689  \B_2^6  \\
&-67434618152750201935011318780202519  \B_2^5 \\
&-321198179800586458917192539207255300566  \B_2^4  \\
&+48653889047087665337770262087432529304260  \B_2^3  \\
&+201755568573332727994250768550108673322824  \B_2^2  \\
&+964913422494459744168713013983828839115872  \B_2  \\
&+1517172696034976766554753118494080984058880=0.
\end{align*}
Of the six roots of this equation only one has absolute value equal to a power of $2$, namely $ \B_2=-2$. With the value of $\B_2$ known, we compute that $c_2=1$. With these values determined we can now compute the expansion of $g$ uniquely to any desired precision. We find 
\[
g(z)=q + q^2 + 2q^3 + 3q^4 + 5q^5 + 2q^6 + 6q^7 + 5q^8 + \ldots ,
\]
and observe that this expansion seems to correspond to that of $E_2^{\. \ch{5,2}, \one{1}}$. We now observe that
\begin{equation}
\Delta_{6,5}=E_4^{\one{1},\. \ch{5,2}} \cdot E_2^{\. \ch{5,2}, \one{1}} \label{eqid}
\end{equation}
seems to hold true, and verify this identity via Corollary \ref{corr1}. With this identity verified, we now have confirmation that the expansion of $g$ must agree exactly with that of $E_2^{\. \ch{5,2}, \one{1}}$, and that \eqref{eqid} is the sole identity for the pair $(4,\ch{5,2})$.

\begin{rmk}
When $N=1$, there are only 12 identities. The proof is given in \cite{duke99} or \cite{ghate00}. Note that the method used here provides a different proof. The identities are given in Table \ref{table4}.
\begin{table} [ht]
\caption{} \label{table4}
\begin{tabular}{|c|c|}
\hline
\rb $\Delta_{16,1}=240 E_4^{\one{1}, \one{1}} \cdot \Delta_{12,1}$ 
& $\Delta_{18,1}=-504 E_6^{\one{1}, \one{1}} \cdot \Delta_{12,1}$ \\
\hline
\rb $\Delta_{20,1}=240 E_4^{\one{1}, \one{1}} \cdot \Delta_{16,1}$
& $\Delta_{20,1}=480 E_8^{\one{1}, \one{1}} \cdot \Delta_{12,1}$ \\
\hline
\rb $\Delta_{22,1}=240 E_4^{\one{1}, \one{1}} \cdot \Delta_{18,1}$ 
& $\Delta_{22,1}=-504 E_6^{\one{1}, \one{1}} \cdot \Delta_{16,1}$ \\
\hline
\rb $\Delta_{22,1}=-264 E_{10}^{\one{1}, \one{1}} \cdot \Delta_{12,1}$ 
& $\Delta_{26,1}=240 E_4^{\one{1}, \one{1}} \cdot \Delta_{22,1}$ \\ 
\hline
\rb $\Delta_{26,1}=-504 E_6^{\one{1}, \one{1}} \cdot \Delta_{20,1}$
& $\Delta_{26,1}=480 E_8^{\one{1}, \one{1}} \cdot \Delta_{18,1}$ \\
\hline
\rb $\Delta_{26,1}=-264 E_{10}^{\one{1}, \one{1}} \cdot \Delta_{16,1}$ 
& $\Delta_{26,1}= -24 E_{14}^{\one{1}, \one{1}} \cdot \Delta_{12,1}$ \\
\hline
\end{tabular}
\end{table}
\end{rmk}

We find nine identities in Case A and 46 identities in Case B. Therefore, there are only 55 identities that conform to Case \ref{case1}. A complete list can be seen among the 61 identities given in Table \ref{table6}.

\begin{rmk}
Note that the obstruction to proving Theorem \ref{thmm5} for the case where $f$ has weight $1$ is our inability to show that there are only finitely many $\psi$ such that $2/B_{1,\psi}$ is an algebraic integer. This in turn is due to the lack of a sufficiently strong lower bound for $|B_{1,\psi}|$, depending on $f_\psi$. We hope to return to this case in a subsequent paper. Note that nowhere in the proof do we require that the weight of $g$ be greater than $1$. The method used here easily covers the case where $g$ has weight $1$.
\end{rmk}

\section{Proof of Theorem \ref{thmm6}} \label{secc6}
We assume that $f(z)=\sum_{n=0}^\i a_f(n)q^n \in \M_k(N,\psi)$, $a_f(0)\neq 0$ and $g(z) =\sum_{n=0}^\i a_g(n) q^n\in \M_l(N,\vp)$, $a_g(0)\neq 0$ are eigenforms, $l \geq k >1$, and that $fg \in \M_{k+l}(N,\psi \vp)$ is also an eigenform. Both $\psi$ and $\vp$ are defined modulo $N$, not necessarily primitive. We split the proof of Theorem \ref{thmm6} into eight propositions (Proposition \ref{PROP1} through \ref{PROP8}). Unlike the proof of Theorem \ref{thmm5}, we use a variety of ad hoc methods to resolve each case. Frequently, for a given case, we will reduce the potential pairs $(l,\vp)$ to a finite number of possibilities, and then determine $(k,\psi)$. Without loss of generality we normalize $f$ and $g$ as usual, so that the coefficient of $q$ is $1$. From Proposition \ref{propp1} we have
\[
f=E_k^{\one{1},\. \psi} \quad \textrm{and} \quad g=E_l^{\one{1},\. \vp}.
\]
Throughout Section \ref{secc6} we use
\[
\a=-\frac{B_{k,\psi}}{2k}, \quad \b=-\frac{B_{l,\vp}}{2l}, \quad \textrm{and} \quad \g=-\frac{B_{k+l,\psi\vp}}{2(k+l)},
\]
along with 
\[
\A_p=p^{k-1}\psi(p) \quad \textrm{and} \quad  \B_p=p^{l-1}\vp(p)
\]
for $p$ prime. Using this notation we write 
\[
f(z)=\a +q+(1+ \A_2)q^2+(1+ \A_3)q^3+(1+ \A_2+ \A_2^2)q^4+\cdots
\]
and
\[
g(z)=\b +q+(1+ \B_2)q^2+(1+ \B_3)q^3+(1+ \B_2+ \B_2^2)q^4+\cdots .
\] 
The product $fg$ has the expansion
\begin{align}
f(z)g(z) &=\a \b+(\a +\b)q+\big(1+\a(1+ \B_2)+\b(1+ \A_2)\big)q^2 \label{eqq17}  \\
& \qquad+\big(\b(1+ \A_3)+\a(1+ \B_3)+2+ \A_2+ \B_2\big)q^3+\cdots .\nonumber
\end{align}
From Proposition \ref{propp1} we know that $fg$ must be a multiple of $E_{k+l}^{\one{1},\. \psi\vp}(z)$. From $\eqref{eqq17}$ we can see that this multiple is $\a +\b$. Therefore the only possible eigenform identities must be of the form
\begin{equation}
E_k^{\one{1},\. \psi} \cdot E_l^{\one{1},\. \vp} = -\left(\frac{B_{k,\psi}}{2k} +\frac{B_{l,\vp}}{2l}\right)  E_{k+l}^{\one{1},\. \psi\vp}. \label{eqq18}
\end{equation}
The expansion of $(\a +\b) E_{k+l}^{\one{1},\. \psi\vp}$ is given by
\begin{align}
(\a +\b)  E_{k+l}^{\one{1},\. \psi\vp}(z)&=\g(\a+\b)+(\a +\b)q \label{eqq19} \\
& \qquad +(\a +\b)(1+2 \A_2 \B_2)q^2 \nonumber \\
& \qquad \qquad +(\a+\b)(1+3 \A_3 \B_3)q^3+\cdots. \nonumber
\end{align}
We equate the coefficients of $q^m$ in $\eqref{eqq17}$ and $\eqref{eqq19}$, for $m=0,2,3,4$ and $8$. After simplifying we obtain, respectively,
\begin{align}
\g(\a +\b)  &=\a \b, \label{eqq20} \\
2 \A_2 \B_2(\a +\b)&= 1+\a  \B_2+\b  \A_2, \label{eqq21} \\
3 \A_3 \B_3(\a +\b)&= 2+\a  \B_3+\b  \A_3+ \A_2+ \B_2, \label{eqq22} \\
4 \A_2^2 \B_2^2(\a +\b) &= 3+ \a \B_2^2 +\b \A_2^2+ \A_2 + \B_2 + \A_2 \B_2+ \A_3+ \B_3, \label{eqq23} \\
8 \A_2^3 \B_2^3(\a +\b)&=4+\a  \B_2^3+\b  \A_2^3+ \A_2^2( \B_2^2+ \B_2+1)+ \B_2^2( \A_2+1) \label{eqq24} \\
&\qquad + \A_2( \A_3+ \A_3 \B_2+ \B_2 \B_3+2 \B_2+ \B_3+2) \nonumber \\
&\qquad \quad+ \B_2( \A_3+ \B_3+2)+ \A_3( \B_5+1)+ \A_5( \B_3+1)\nonumber \\
& \qquad \quad \quad +  \A_7+ \B_3+ \B_5+ \B_7, \nonumber
\end{align}
(to obtain \eqref{eqq23} we subtracted \eqref{eqq22}, and to obtain \eqref{eqq24} we subtracted \eqref{eqq22} and \eqref{eqq23}). We now combine $\eqref{eqq21}$, $\eqref{eqq22}$ and $\eqref{eqq23}$ to eliminate $\a$ and $ \A_3$, resulting in an equation determining $ \A_2$ in terms of $\b$, $ \B_2$ and $ \B_3$. After dividing by the factor $2 \A_2-1$ we have
\begin{equation}
j_3 \A_2^3+j_2 \A_2^2+j_1 \A_2+j_0=0, \label{eqq25}
\end{equation}
where
\begin{align*}
j_3&= \b^2\left(4 \B_2^2-2 \B_2(3 \B_3+1)+3 \B_3\right),\\
j_2&=\b^2\left(-4 \B_2^3+12 \B_2^2 \B_3+ \B_2(1-6 \B_3)\right)\\
& \qquad +\b\left(2 \B_2^2- \B_2(9 \B_3+2)+6 \B_3\right)+2 \B_2,\\
j_1&= \b^2\left(2 \B_2^3(1-3 \B_3)+ \B_2^2(3 \B_3-1)\right) \\
& \qquad +\b\left( \B_2^2(9 \B_3-1)- \B_2(10 \B_3+3)+3 \B_3^2+7 \B_3\right)  \\
& \qquad \qquad +2 \B_2^2+3  \B_2(1- \B_3)+3 \B_3,  \\
j_0&= \b  \B_2(2-3 \B_3^2-5 \B_3)- \B_2^2-2 \B_2+3 \B_3^2+7 \B_3.
\end{align*}

\begin{prop} \label{PROP1}
Assume that $N=1$. There are only four eigenform identities of the form \eqref{eqq18}. They are given in Table \ref{table5}.
\begin{table}[ht]
\caption{} \label{table5}
\begin{tabular}{|c|c|}
\hline
\rb $E_8^{\one{1}, \one{1}}=120E_4^{\one{1}, \one{1}} \cdot E_4^{\one{1}, \one{1}}$ & 
$E_{10}^{\one{1}, \one{1}}=5040/11 E_4^{\one{1}, \one{1}} \cdot E_6^{\one{1}, \one{1}}$ \\
\hline
\rb $E_{14}^{\one{1}, \one{1}}=2640 E_4^{\one{1}, \one{1}} \cdot E_{10}^{\one{1}, \one{1}}$ & 
$E_{14}^{\one{1}, \one{1}}=10080 E_6^{\one{1}, \one{1}} \cdot E_8^{\one{1}, \one{1}}$ \\
\hline
\end{tabular}
\end{table}
\end{prop}
\begin{proof}
See either Duke \cite{duke99} or Ghate \cite{ghate00}. Note that the normalization we use here is different than what Duke and Ghate use. Writing 
\[
E_k(z)=-\frac{2k}{B_{k,\one{1}}}E_k^{\one{1}, \one{1}}(z)= 1-\frac{2k}{B_k}\sum_{n=1}^\i \sigma_{k-1}(n)q^n
\]
we obtain the familiar identities
\[
E_8=E_4^2, \quad E_{10}=E_4 E_6, \quad E_{14}=E_4 E_{10} \quad \textrm{and} \quad E_{14}=E_6 E_8.
\]
For an elementary but less elegant proof than the one given in \cite{ghate00}, one can use \eqref{eqq20} and argue that  
\[
\frac{(k+l)B_k B_l }{l B_k+k B_l}=B_{k+l}
\]
with $l \geq k$, is only true for the pairs of weights 
\[
(k,l)\in \{(4,4),(4,6),(4,10),(6,8)\}.  \qedhere
\]
\end{proof}

\begin{prop} \label{PROP2}
Assume that $2|N$. There are no eigenform identities of the form \eqref{eqq18}.
\end{prop}
\begin{proof}
If $2|N$, then $ \A_2= \B_2=0$, and $\eqref{eqq21}$ would result in a contradiction.
\end{proof}
For the remainder of this section we assume that the level $N$ is odd. We have
\[
| \A_2|=2^{k-1} \quad \textrm{and} \quad  | \B_2|=2^{l-1}.
\]
Thus equation \eqref{eqq25} gives us a straightforward way to determine if an identity is possible for a fixed pair $(l,\vp)$. Observe that \eqref{eqq25} must have a root $\A_2$ satisfying $|\A_2|=2^{k-1}$ in order for an identity to exist with $g=E_l^{\one{1},\. \vp}$. We substitute the values
\[
\b=-\frac{B_{l,\vp}}{2l}, \quad  \B_2=2^{l-1}\vp(2) \quad \textrm{and} \quad  \B_3=3^{l-1}\vp(3)
\]
into \eqref{eqq25}, and determine if there are any valid roots. Since an approximation is sufficient, this computation is straightforward. If $\eqref{eqq25}$ has no valid roots, there are no identities arising from the pair $(l,\vp)$. If we find a valid root $ \A_2$, then the values of $\a$ and $ \A_3$ are determined by \eqref{eqq20} and \eqref{eqq21}, respectively. Further values of $ \A_p$ are uniquely determined by equating further coefficients of $\eqref{eqq17}$ with $\eqref{eqq19}$. Thus the Fourier expansion of $f$ will be uniquely determined. With this expansion known we can determine if the expansion of $f$ corresponds to that of an eigenform.

\begin{prop} \label{PROP3}
Assume that $3|N$. There is only one eigenform identity of the form \eqref{eqq18}, given below:
\begin{equation}
E_5^{\one{1},\. \ch{3}}=-36E_2^{\one{1},\. \one{3}} \cdot E_3^{\one{1},\. \ch{3}}. \label{eqq26}
\end{equation}
\end{prop}
\begin{proof}
If $3$ divides the level, then $ \A_3= \B_3=0$, and $\eqref{eqq22}$ reduces to 
\begin{equation}
-2= \A_2+ \B_2. \label{eqq27}
\end{equation}
Since $| \A_2|$ and $| \B_2|$ are powers of $2$, there are only two solutions to $\eqref{eqq27}$. Either $ \A_2= \B_2=-1$ (which we exclude since we assume $l\geq k > 1$) or $ \A_2=2$ and $ \B_2=-4$, in which case $k=2$ and $l=3$. We set $ \A_2=2$ and $ \B_2=-4$ in equations $\eqref{eqq21}$ and $\eqref{eqq23}$ (along with $ \A_3= \B_3=0$). Solving for $\a$ and $\b$ yields $\a=1/12$ and $\b=-1/9$. Now observe that $\eqref{eqq26}$ is true by applying Corollary \ref{corr1}, and that $ \A_p$ and $ \B_p$ (for $p>3$ prime) are uniquely determined by equating further coefficients of $E_2^{\one{1},\. \psi} \cdot E_3^{\one{1},\. \vp}$ with  $(\frac{1}{12} -\frac{1}{9})  E_5^{\one{1},\. \psi\vp}$. Therefore $\eqref{eqq26}$ is the only identity when $3$ divides the level.
\end{proof}
For the remainder of the proof we now assume that the level $N$ is relatively prime to both $2$ and $3$ (and therefore that $f_\psi$ and $f_\vp$ are prime to $2$ and $3$ as well). We will make use of the fact that we can assume $f_{\vp}\geq 5$ if $\vp$ is not trivial (i.e., $\vp\neq \one{N}$).

\begin{prop} \label{PROP4}
Assume that $l=k\geq 3$. There are no eigenform identities of the form \eqref{eqq18}.
\end{prop}
\begin{proof}
We first assume that both $\psi$ and $\vp$ are non-trivial characters. Given $\a=-B_{k,\psi}/(2k)$ and $\b=-B_{l,\vp}/(2l)$, we have from \eqref{eqq8} that 
\[
4\a\equiv 4\b \equiv 0 \Mod{2}.
\]
From $\eqref{eqq21}$ we have
\begin{equation}
2^{2k-1}\psi(2)\vp(2)(\a+\b)-2^{k-1}(\vp(2)\a+\psi(2)\b)=1. \label{eqq28}
\end{equation}
Since $k\geq 3$, the left side of $\eqref{eqq28}$ is equivalent to $0$ mod $2$, which results in a contradiction. Thus there are no identities involving non-trivial characters.  If $\psi$ (or $\vp)$ is the trivial character, say $\psi=\one{N}$, then $k$ must be even and
\[
\a=-\frac{B_k}{2k}\prod_{p\,|N}(1-p^{k-1}).
\]
We define $\ord_2(n)$ to be the largest integer $e$ such that $2^e$ divides $n$. We want to show that $\ord_2(2^{k-1} \a) \geq 1$. Observe that $B_k$ will always have exactly one factor of $2$ in the denominator (from the von Staudt-Clausen theorem) and $1-p^{k-1}$ will be even since $p \neq 2$. Therefore $\ord_2(2^{k-1} \a) \geq k-2-\ord_2(k)$. For $k\geq 3$ we have $k-2-\ord_2(k)\geq 1.$ So again the left side of \eqref{eqq28} is $0$ mod $2$, yielding a contradiction. Thus there are no identities.
\end{proof}

\begin{prop} \label{PROP5}
Assume that $l=k=2$. There is only one eigenform identity of the form \eqref{eqq18}, given below:
\begin{equation}
E_4^{\one{1},\. \ch{5,2}}=-30E_2^{\one{1},\. \one{5}} \cdot E_2^{\one{1},\. \ch{5,2}}. \label{eqq29}
\end{equation}
\end{prop}
\begin{proof}
From \eqref{eqq21} we have 
\begin{equation}
2 \a\vp(2) ( 4 \psi(2)-1) +2 \b\psi(2) ( 4 \vp(2)-1)=1. \label{eqq30}
\end{equation}
First we assume that $\psi$ and $\vp$ are non-trivial. Since both characters must be even we have from \eqref{eqq9} that $2\a \equiv 2\b \equiv 0 \Mod{2}$. Therefore the left side of \eqref{eqq30} is $0$ mod $2$, and we have a contradiction. Now we assume that exactly one of either $\psi$ or $\vp$ is the trivial character. Without loss of generality we take $\psi=\one{N}$. Using
\[
\a=-\frac{1}{24}\prod_{p\,|N}(1-p)
\]
and \eqref{eqq30} we obtain
\[
-\vp(2)\prod_{p\,|N}(1-p)=4+8\b( 1-4\vp(2)).
\]
Therefore $M\vp(2) \equiv 4 \Mod{8}$ for some $M\in \Z$, and thus $\vp(2) =\pm 1$. If we take $\vp(2)=1$, then \eqref{eqq25} yields $\a=11/6$. If we take $\vp(2)=-1$ then \eqref{eqq25} yields $\a=1/6$ or $\a=23/52$. Observe that $23/52$ cannot be the value of $\a$. If $\a=11/6$ then $N=2\cdot 3 \cdot 23$, which is excluded. The value $\a=1/6$ corresponds to the identity \eqref{eqq29}. If both $\psi$ and $\vp$ are trivial, $\psi=\vp=\one{N}$, then \eqref{eqq30} yields $\a=1/12$, which implies that $N=3$. Thus there are no further identities.
\end{proof}

For the remainder of this proof we now assume that $k$ is strictly less than $l$. Recall that on page \pageref{tagg3} we defined
\[
C(l,f_\vp)=\frac{\zeta(2l)(l-1)!f_\vp^{l-1/2}}{\zeta(l)(2\pi)^l}. 
\]
We will now proceed to define the function $D(k,l)$ so that
\[
C(l,f_{\vp}) \leq \left|\frac{B_{l,\vp_0}}{2l}\right| \leq \left|\frac{B_{l,\vp}}{2l}\right| \leq D(k,l).
\]
This allows us to exclude all cases where $C(l,f_\vp)>D(k,l)$. We want to obtain an equation that determines $\b$ in terms of $ \A_p$'s and $ \B_p$'s. To do so, we combine $\eqref{eqq21}$ and $\eqref{eqq24}$ to eliminate $\a$. Upon dividing by $ \B_2$ (noting that $ \B_2$ is non-zero) we have
\[
c_1\b=c_0
\]
where
\[
c_1=8 \A_2^3 \B_2^3-2 \A_2 \B_2^3-8 \A_2^4 \B_2^2+ \A_2 \B_2^2+2 \A_2^4- \A_2^3
\]
and
\[
c_0=d_3 \A_2^3+d_2 \A_2^2+d_1 \A_2+d_0,
\]
with
\begin{align*}
d_3&=-6 \B_2^2 + 2 \B_2 + 2,\\
d_2&= \B_2^2 + 2 \A_3 \B_2 + 2 \B_2 \B_3 + 3 \B_2 + 2 \A_3 +2 \B_3 + 3,\\
d_1&= \B_2^2 +  \A_3 \B_2 +  \B_2 \B_3 + 2 \B_2 + 2 \A_3 \B_5 +  \A_3 + 2 \A_5 \B_3 \\
& \qquad +2 \A_5 + 2 \A_7 +  \B_3 + 2 \B_5 + 2 \B_7 + 6,  \\
d_0&=- \A_3 \B_2 - \B_2 \B_3 - 2 \B_2 - \A_3 \B_5 -  \A_3- \A_5 \B_3 \\
& \qquad -  \A_5 - \A_7 - \B_3 - \B_5 - \B_7 - 4.
\end{align*}
We write $\b=c_0/c_1$. We want an upper bound for $|\b|$. In order to determine such a bound, we obtain a lower bound on $c_1$ and an upper bound on $c_0$. Recall that $| \A_p|=p^{k-1}$ and $| \B_p|=p^{l-1}$. Using the reverse triangle inequality on $c_1$ we obtain
\[
|c_1|\geq 8| \A_2^3 \B_2^3|-2| \A_2 \B_2^3|-8| \A_2^4 \B_2^2|-| \A_2 \B_2^2|-2| \A_2^4|-| \A_2^3|.
\]
Therefore we have
\[
|c_1| \geq 2^{k-3}(8^l 4^k-8^k 4^l-8^l-8^k-4^l-4^k).
\]
Note that
\[ 8^l 4^k-8^k 4^l-8^l-8^k-4^l-4^k>0\] 
when $l>k>1$. We define $e_i$ as follows,
\begin{align*}
e_3&=6 \cdot 2^{2l-2}+2^{l}+2,\\
e_2&=2^{2l-2}+3^{k-1}2^l+3^{l-1}2^l+3 \cdot 2^{l-1}+2 \cdot 3^{k-1}+2 \cdot 3^{l-1}+3,\\
e_1&=2^{2l-2}+3^{k-1}2^{l-1}+2^{l-1}3^{l-1}+2^l+2 \cdot 3^{k-1}5^{l-1}\\
 & \qquad +3^{k-1}+2 \cdot 5^{k-1}3^{l-1}+2 \cdot 5^{k-1}+2 \cdot 7^{k-1}\\
 & \qquad \qquad +3^{l-1}+2 \cdot 5^{l-1}+ 2 \cdot 7^{l-1}+6, \\
e_0&=3^{k-1}2^{l-1}+2^{l-1}3^{l-1}+2^{l}+3^{k-1}5^{l-1}+3^{k-1}+5^{k-1}3^{l-1}\\
& \qquad +5^{k-1}+7^{k-1}+3^{l-1}+5^{l-1}+7^{l-1}+4,  
\end{align*}
so that $|d_i| \leq e_i$ for $0\leq i \leq 3$. Therefore we have
\[
|c_0|\leq 2^{3k-3}e_3+2^{2k-2}e_2+2^{k-1}e_1+e_0. 
\]
We define
\[
D(k,l):=\frac{4^k e_3+2^{k+1}e_2+4e_1+2^{3-k}e_0}
{8^l 4^k-8^k 4^l-8^l-8^k-4^l-4^k},
\]
so that for $l>k>1$ we have
\[
|\b|=\left|\frac{B_{l,\vp}}{2l}\right| \leq D(k,l).
\]

\begin{prop} \label{PROP6}
Assume that $l\geq 5$ and $1<k<l$. There are no eigenform identities of the form \eqref{eqq18}.
\end{prop}
\begin{proof}
When $l\geq 5$, we have $D(k,l)<1$ for any $1 < k < l$, and therefore $|\b|<1$. If $\vp \neq \one{N}$, then $f_{\vp}\geq 5$. We have $C(l,f_{\vp})>1$ when $l \geq 5$ and $f_{\vp}\geq 5$. Thus there are no eigenform identities when $\vp$ is non-trivial. If $\vp=\one{N}$, then there exists at least one prime $p\geq 5$ that divides $N$. From Proposition \ref{lemm1} and equation \eqref{eqq6} we have that $|\b| \geq C(l,1)(5^{l-1}-1)$. Observing that $C(l,1)(5^{l-1}-1)>1$ for $l \geq 5$, we determine there are no identities.
\end{proof}

\begin{prop} \label{PROP7}
Assume $l= 4$ and $k=2$ or $k=3$. There are no eigenform identities of the form \eqref{eqq18}.
\end{prop}
\begin{proof}
We have $D(2,4) < 0.53$ and $D(3,4) < 0.49$, and that $C(4,f_{\vp})>0.99$ for $f_{\vp}\geq 5$. So the only possibility for $\vp$ is $\one{N}$. Let $p$ be a prime dividing $N$. If $p\geq 7$ then $|\b|>C(4,1)(7^3-1)>1.22$. So there is only one case to consider, namely $\vp=\one{5}$. We have $\b=-31/60$, $\B_2=8$ and $\B_3=27$. Using these values in \eqref{eqq25} we obtain
\[
24986 \A_2^4 - 542765 \A_2^3 + 3050832 \A_2^2 - 2533840 \A_2 + 570496=0.
\]
Since $| \A_2|=2^{k-1}$, there is only one valid root, namely $ \A_2=8$. This implies that $k=4$, which we have excluded. Thus there are no identities.
\end{proof}

\begin{prop} \label{PROP8}
Assume that $l= 3$ and $k=2$. There are no eigenform identities of the form \eqref{eqq18}.
\end{prop}
\begin{proof}
We proceed along the same lines as in Proposition \ref{PROP7}. We have $D(2,3) < 1.59$. Since $l$ is odd, $\vp \neq \one{N}$. We have $C(3,f_{\vp})>1.64$ for $f_{\vp}\geq 9$. Therefore $f_{\vp}$ must be either 5 or 7. If $\vp$ is not primitive, then there is a prime $p \geq 5$ dividing $N$ but not $f_{\vp}$. In this case we would have $|\b|>C(3,5)(5^2-1)> 9.15$. Therefore, we need only check the cases where $\vp$ is an odd primitive character modulo $5$ or $7$. When $\vp=\ch{5,1}$, $\eqref{eqq25}$ has one valid root, $ \A_2=i$. This corresponds to the identity 
\[
E_4^{\one{1},\. \ch{5,2}}=(-5+5i)E_1^{\one{1},\. \ch{5,1}} \cdot E_3^{\one{1},\. \ch{5,1}},
\]
which we exclude since $k=1$. When $\vp=\ch{5,3}$, $\eqref{eqq25}$ again has only one valid root, $ \A_2=-i$. This corresponds to the excluded conjugate identity 
\[E_4^{\one{1},\. \ch{5,2}}=(-5-5i)E_1^{\one{1},\. \ch{5,3}} \cdot E_3^{\one{1},\. \ch{5,3}}.\] 
For the three cases where $\vp$ is an odd character modulo $7$, $\eqref{eqq25}$ has no valid roots. Thus there are no identities.
\end{proof}
Propositions \ref{PROP1} through \ref{PROP8} cover all possibilities for Case \ref{case2}, with $l \geq k>1$.

\section{Eigenform Product Identities} \label{secc7}
We briefly explain the notation used in Tables \ref{table6} and \ref{table8}. The Dirichlet characters $\chi_{p,j}$ are defined on page \pageref{tagg1}. The Eisenstein series $E_k^{\. \psi,\. \vp}$ is defined in equation \eqref{eqq4}. See page \pageref{tagg2} for the definition of the normalized cuspidal eigenforms $\Delta_{k,N}$, $\Delta_{k,N,\chi}$ and $\Phi_{k,N}$. Note that each occurrence of the function $\Phi_{k,N}$ denotes an eigenform product identity \emph{not} forced by dimensional considerations. Because $\Phi_{k,N}$ is not specified by $k$ and $N$ (unlike $\Delta_{k,N}$ and $\Delta_{k,N,\chi}$), we provide a partial Fourier expansion of each $\Phi_{k,N}$ used, for identification purposes. Note that $\Phi_{4,10}$ is the unique normalized newform in $\S_{4}(\G_0(10))$. The other $\Phi_{k,N}$ occur in spaces when the dimension of $\S_k^{{\textrm{new}}}(\G_0(N))$ is also greater than one.

\label{tagg4}
\begin{align*} 
\Phi_{10,3}(z) &= q + 18q^2 + 81q^3 - 188q^4 - 1530q^5 + 1458q^6 + 9128q^7 + \cdots  \\
\Phi_{8,5}(z)  &= q - 14q^2 - 48q^3 + 68q^4 + 125q^5 + 672q^6 - 1644q^7 + 840q^8 + \cdots \\
\Phi_{10,5}(z) &= q - 8q^2 - 114q^3 - 448q^4 - 625q^5 + 912q^6 + 4242q^7 + 7680q^8 + \cdots \\
\Phi_{6,7}(z)  &= q - 10q^2 - 14q^3 + 68q^4 - 56q^5 + 140q^6 - 49q^7 - 360q^8 + \cdots \\
\Phi_{8,7}(z)  &= q - 6q^2 - 42q^3 - 92q^4 - 84q^5 + 252q^6 + 343q^7 + 1320q^8 + \cdots \\
\Phi_{4,10}(z) &= q + 2q^2 - 8q^3 + 4q^4 + 5q^5 - 16q^6 - 4q^7 + 8q^8 + 37q^9 + 10q^{10} + \cdots \\
\Phi_{6,10}(z) &= q - 4q^2 - 26q^3 + 16q^4 - 25q^5 + 104q^6 - 22q^7 - 64q^8 + 433q^9 +  \cdots \\
\Phi_{4,15}(z) &= q + q^2 + 3q^3 - 7q^4 + 5q^5 + 3q^6 - 24q^7 - 15q^8 + 9q^9 + 5q^{10}+ \cdots \\
\Phi_{6,15}(z) &= q - 2q^2 - 9q^3 - 28q^4 - 25q^5 + 18q^6 - 132q^7 + 120q^8 + 81q^9+  \cdots \\ 
\Phi_{4,21}(z) &= q - 3q^2 - 3q^3 + q^4 - 18q^5 + 9q^6 + 7q^7 + 21q^8+ 9q^9+54q^{10}+ \cdots \\
\end{align*}

\begin{rmk} \label{remm2}
There are a few interesting observations to make regarding the identities in Table \ref{table6}. There are nine identities where the product of two Eisenstein series is again an Eisenstein series. There are 28 identities where the product of two Eisenstein series is a cusp form. And the remaining 24 identities are where the product of an Eisenstein series and a cusp form is a cusp form. One might think that the product of two Eisenstein series would not be a cusp form very often, but we find that this is the most frequent kind of identity. Of the nine instances of the product of two Eisenstein series being an Eisenstein series, six correspond to Case \ref{case2}, and three to Case \ref{case1}. Out of the 61 identities, 54 occur at square-free level, and seven at square levels ($N=4$, $8$ and $9$).

Note that while both $E_3^{\one{1},\. \ch{4}}$ and $\Delta_{5,4,\ch{4}}$ are eigenforms of level $4$, their product $\Delta_{8,2}$ is an eigenform of level $2$. We list the identity 
\[
\Delta_{8,2}=-4E_3^{\one{1},\. \ch{4}} \cdot \Delta_{5,4,\ch{4}}
\]
under the level $N=4$, since $4$ is the smallest level that all three forms exist at. Also of note is the fact that the three identities with the largest levels ($N=10$, $15$ and $21$) are not forced by dimensional considerations. 
\end{rmk}

\begin{table}[ht] 
\caption{Eigenform product identities} \label{table6}
\begin{tabular}{|c|c|c|l|}
\hline
Level & Weight & Nebentypus & \multicolumn{1} {c|} {Identity} \\
\hline
\hline
\rb $N=1$ & $k=8$ & $\chi=\one{1}$ & $E_8^{\one{1},\one{1}}=120E_4^{\one{1}, \one{1}} \cdot E_4^{\one{1}, \one{1}}$ \\
\hline
\rb $N=1$ & $k=10$ & $\chi=\one{1}$ & $E_{10}^{\one{1}, \one{1}}=5040/11 E_4^{\one{1}, \one{1}} \cdot E_6^{\one{1}, \one{1}}$ \\
\hline
\rb $N=1$ & $k=14$ & $\chi=\one{1}$ & $E_{14}^{\one{1}, \one{1}}=2640 E_4^{\one{1}, \one{1}} \cdot E_{10}^{\one{1}, \one{1}}$ \\
\hline
\rb $N=1$ & $k=14$ & $\chi=\one{1}$ & $E_{14}^{\one{1}, \one{1}}=10080 E_6^{\one{1}, \one{1}} \cdot E_8^{\one{1}, \one{1}}$ \\
\hline
\rb $N=1$ & $k=16$ & $\chi=\one{1}$ & $\Delta_{16,1}=240 E_4^{\one{1}, \one{1}} \cdot \Delta_{12,1}$ \\
\hline
\rb $N=1$ & $k=18$ & $\chi=\one{1}$ & $\Delta_{18,1}=-504 E_6^{\one{1}, \one{1}} \cdot \Delta_{12,1}$ \\
\hline
\rb $N=1$ & $k=20$ & $\chi=\one{1}$ & $\Delta_{20,1}=240 E_4^{\one{1}, \one{1}} \cdot \Delta_{16,1}$ \\
\hline
\rb $N=1$ & $k=20$ & $\chi=\one{1}$ & $\Delta_{20,1}=480 E_8^{\one{1}, \one{1}} \cdot \Delta_{12,1}$ \\
\hline
\rb $N=1$ & $k=22$ & $\chi=\one{1}$ & $\Delta_{22,1}=240 E_4^{\one{1}, \one{1}} \cdot \Delta_{18,1}$ \\
\hline
\rb $N=1$ & $k=22$ & $\chi=\one{1}$ & $\Delta_{22,1}=-504 E_6^{\one{1}, \one{1}} \cdot \Delta_{16,1}$ \\
\hline
\rb $N=1$ & $k=22$ & $\chi=\one{1}$ & $\Delta_{22,1}=-264 E_{10}^{\one{1}, \one{1}} \cdot \Delta_{12,1}$ \\
\hline
\rb $N=1$ & $k=26$ & $\chi=\one{1}$ & $\Delta_{26,1}=240 E_4^{\one{1}, \one{1}} \cdot \Delta_{22,1}$ \\
\hline
\rb $N=1$ & $k=26$ & $\chi=\one{1}$ & $\Delta_{26,1}=-504 E_6^{\one{1}, \one{1}} \cdot \Delta_{20,1}$ \\
\hline
\rb $N=1$ & $k=26$ & $\chi=\one{1}$ & $\Delta_{26,1}=480 E_8^{\one{1}, \one{1}} \cdot \Delta_{18,1}$ \\
\hline
\rb $N=1$ & $k=26$ & $\chi=\one{1}$ & $\Delta_{26,1}=-264 E_{10}^{\one{1}, \one{1}} \cdot \Delta_{16,1}$ \\
\hline
\rb $N=1$ & $k=26$ & $\chi=\one{1}$ & $\Delta_{26,1}= -24 E_{14}^{\one{1}, \one{1}} \cdot \Delta_{12,1}$ \\
\hline
\hline
\rb $N=2$ & $k=6$ & $\chi=\one{2}$ & $E_6^{\one{2}, \one{1}}=24E_2^{\one{1}, \one{2}} \cdot E_4^{\one{2}, \one{1}}$ \\
\hline
\rb $N=2$ & $k=10$ & $\chi=\one{2}$ & $\Delta_{10,2}=24E_2^{\one{1}, \one{2}}  \cdot \Delta_{8,2}$ \\
\hline
\hline
\rb $N=3$ & $k=5$ & $\chi=\ch{3}$ & $E_5^{\one{1},\, \ch{3}}=-36E_2^{\one{1}, \one{3}} \cdot E_3^{\one{1},\, \ch{3}}$ \\
\hline
\rb $N=3$ & $k=5$ & $\chi=\ch{3}$ & $E_5^{\ch{3}, \one{1}}=12E_2^{\one{1}, \one{3}}  \cdot E_3^{\, \ch{3}, \one{1}}$ \\
\hline
\rb $N=3$ & $k=6$ & $\chi=\one{3}$ & $\Delta_{6,3}=-9E_3^{\one{1},\, \ch{3}} \cdot E_3^{\, \ch{3}, \one{1}}$ \\
\hline
\rb $N=3$ & $k=7$ & $\chi=\ch{3}$ & $\Delta_{7,3,\ch{3}}=-9E_3^{\one{1},\, \ch{3}} \cdot E_4^{\one{3}, \one{1}}$ \\
\hline
\rb $N=3$ & $k=8$ & $\chi=\one{3}$ & $\Delta_{8,3}=-9E_3^{\one{1},\, \ch{3}} \cdot E_5^{\, \ch{3}, \one{1}}$ \\
\hline
\rb $N=3$ & $k=8$ & $\chi=\one{3}$ & $\Delta_{8,3}=3E_5^{\one{1},\, \ch{3}} \cdot E_3^{\, \ch{3}, \one{1}}$ \\
\hline
\rb $N=3$ & $k=8$ & $\chi=\one{3}$ & $\Delta_{8,3}=12E_2^{\one{1}, \one{3}}  \cdot \Delta_{6,3}$ \\
\hline
\rb $N=3$ & $k=10$ & $\chi=\one{3}$ & $\Phi_{10,3}=3E_5^{\one{1},\, \ch{3}} \cdot E_5^{\, \ch{3}, \one{1}}$ \\
\hline
\rb $N=3$ & $k=10$ & $\chi=\one{3}$ & $\Phi_{10,3}=12E_2^{\one{1}, \one{3}}  \cdot \Delta_{8,3}$ \\
\hline
\hline
\rb $N=4$ & $k=5$ & $\chi=\ch{4}$ & $\Delta_{5,4,\ch{4}}=-4E_3^{\one{1},\, \ch{4}} \cdot E_2^{\one{2}, \one{2}}$ \\
\hline
\rb $N=4$ & $k=6$ & $\chi=\one{2}$ & $\Delta_{6,4}=-4E_3^{\one{1},\, \ch{4}} \cdot E_3^{\, \ch{4}, \one{1}}$ \\
\hline
\rb $N=4$ & $k=8$ & $\chi=\one{2}$ & $\Delta_{8,2}=-4E_3^{\one{1},\, \ch{4}} \cdot \Delta_{5,4,\ch{4}}$ \\
\hline
\end{tabular}
\end{table}

\begin{table} [ht] \addtocounter{table}{-1}
\caption{Eigenform product identities (continued)} 
\begin{tabular}{|c|c|c|l|}
\hline
Level & Weight & Nebentypus & \multicolumn{1} {c|} {Identity}  \\
\hline
\hline
\rb $N=5$ & $k=4$ & $\chi=\one{5}$ & $\Delta_{4,5}=-5E_2^{\one{1},\, \ch{5,2}} \cdot E_2^{\, \ch{5,2}, \one{1}}$ \\
\hline
\rb $N=5$ & $k=4$ & $\chi=\ch{5,2}$ & $E_4^{\one{1},\, \ch{5,2}}=-30E_2^{\one{1}, \one{5}} \cdot E_2^{\one{1},\, \ch{5,2}} $ \\
\hline
\rb $N=5$ & $k=4$ & $\chi=\ch{5,2}$ & $E_4^{\, \ch{5,2}, \one{1}}=6E_2^{\one{1}, \one{5}}  \cdot E_2^{\, \ch{5,2}, \one{1}}$ \\
\hline
\rb $N=5$ & $k=5$ & $\chi=\ch{5,3}$ & $\Delta_{5,5,\ch{5,3}}=-5E_2^{\one{1},\, \ch{5,2}} \cdot E_3^{\, \ch{5,1}, \one{1}}$ \\
\hline
\rb $N=5$ & $k=5$ & $\chi=\ch{5,1}$ & $ \Delta_{5,5,\ch{5,1}}=-5E_2^{\one{1},\, \ch{5,2}} \cdot E_3^{\, \ch{5,3}, \one{1}}$ \\
\hline
\rb $N=5$ & $k=5$ & $\chi=\ch{5,3}$ & $ \Delta_{5,5,\ch{5,3}}=(-2+i)E_3^{\one{1},\, \ch{5,1}} \cdot E_2^{\, \ch{5,2}, \one{1}}$ \\
\hline
\rb $N=5$ & $k=5$ & $\chi=\ch{5,1}$ & $ \Delta_{5,5,\ch{5,1}}=(-2-i)E_3^{\one{1},\, \ch{5,3}} \cdot E_2^{\, \ch{5,2}, \one{1}}$ \\
\hline
\rb $N=5$ & $k=6$ & $\chi=\one{5}$ & $\Delta_{6,5}=6E_2^{\one{1}, \one{5}} \cdot \Delta_{4,5}$ \\
\hline
\rb $N=5$ & $k=6$ & $\chi=\one{5}$ & $ \Delta_{6,5}=-5E_2^{\one{1},\, \ch{5,2}} \cdot E_4^{\, \ch{5,2}, \one{1}}$ \\
\hline
\rb $N=5$ & $k=6$ & $\chi=\one{5}$ & $ \Delta_{6,5}=(-2+i)E_3^{\one{1},\, \ch{5,1}} \cdot E_3^{\, \ch{5,3}, \one{1}}$ \\
\hline
\rb $N=5$ & $k=6$ & $\chi=\one{5}$ & $ \Delta_{6,5}=(-2-i)E_3^{\one{1},\, \ch{5,3}} \cdot E_3^{\, \ch{5,1}, \one{1}}$ \\
\hline
\rb $N=5$ & $k=6$ & $\chi=\one{5}$ & $ \Delta_{6,5}=E_4^{\one{1},\, \ch{5,2}} \cdot E_2^{\, \ch{5,2}, \one{1}}$ \\
\hline
\rb $N=5$ & $k=10$ & $\chi=\one{5}$ & $\Phi_{10,5}=6E_2^{\one{1}, \one{5}} \cdot \Phi_{8,5}$ \\
\hline
\hline
\rb $N=6$ & $k=4$ & $\chi=\one{6}$ & $\Delta_{4,6}=-E_2^{\one{1},\, \ch{12}} \cdot E_2^{\, \ch{3},\, \ch{4}}$ \\
\hline
\hline
\rb $N=7$ & $k=4$ & $\chi=\one{7}$ & $\Delta_{4,7}=(-\z{6}-2)E_2^{\one{1},\, \ch{7,2}} \cdot E_2^{\, \ch{7,4}, \one{1}}$ \\
\hline
\rb $N=7$ & $k=4$ & $\chi=\one{7}$ & $\Delta_{4,7}=(\z{6}+3)E_2^{\one{1},\, \ch{7,4}} \cdot E_2^{\, \ch{7,2}, \one{1}}$ \\
\hline
\rb $N=7$ & $k=4$ & $\chi=\ch{7,4}$ & $\Delta_{4,7,\ch{7,4}}=(-\z{6}-2)E_2^{\one{1},\, \ch{7,2}} \cdot E_2^{\, \ch{7,2}, \one{1}}$ \\
\hline
\rb $N=7$ & $k=4$ & $\chi=\ch{7,2}$ & $\Delta_{4,7,\ch{7,2}}=(\z{6}+3)E_2^{\one{1},\, \ch{7,4}} \cdot E_2^{\, \ch{7,4}, \one{1}}$ \\
\hline
\rb $N=7$ & $k=5$ & $\chi=\ch{7,3}$ & $\Delta_{5,7,\ch{7,3}}=4E_2^{\one{1}, \one{7}} \cdot \Delta_{3,7,\ch{7,3}}$ \\
\hline
\rb $N=7$ & $k=5$ & $\chi=\ch{7,3}$ & $\Delta_{5,7,\ch{7,3}}=(-\z{6}-2)E_2^{\one{1}, \, \ch{7,2}} \cdot E_3^{\, \ch{7,1}, \one{1}}$ \\
\hline
\rb $N=7$ & $k=5$ & $\chi=\ch{7,3}$ & $\Delta_{5,7,\ch{7,3}}=(\z{6}+3)E_2^{\one{1},\, \ch{7,4}} \cdot E_3^{\, \ch{7,5}, \one{1}}$ \\
\hline
\rb $N=7$ & $k=5$ & $\chi=\ch{7,3}$ & $\Delta_{5,7,\ch{7,3}}=(\z{6}-1)E_3^{\one{1},\, \ch{7,1}} \cdot E_2^{\, \ch{7,2}, \one{1}}$ \\
\hline
\rb $N=7$ & $k=5$ & $\chi=\ch{7,3}$ & $\Delta_{5,7,\ch{7,3}}=-\z{6}E_3^{\one{1},\, \ch{7,5}} \cdot E_2^{\, \ch{7,4}, \one{1}}$ \\
\hline
\rb $N=7$ & $k=8$ & $\chi=\one{7}$ & $\Phi_{8,7}=4E_2^{\one{1}, \one{7}} \cdot \Phi_{6,7}$ \\
\hline
\hline
\rb $N=8$ & $k=4$ & $\chi=\one{2}$ & $\Delta_{4,8}=-2E_2^{\one{1},\, \ch{8,2}} \cdot E_2^{\, \ch{8,2}, \one{1}}$ \\
\hline
\rb $N=8$ & $k=5$ & $\chi=\ch{4}$ &  $\Delta_{5,4,\ch{4}}=-2E_2^{\one{1},\, \ch{8,2}} \cdot \Delta_{3,8,\ch{8,1}}$ \\
\hline
\hline
\rb $N=9$ & $k=4$ & $\chi=\one{3}$ & $\Delta_{4,9}=(\z{6}-2)E_2^{\one{1},\, \ch{9,2}} \cdot E_2^{\, \ch{9,4}, \one{1}}$ \\
\hline
\rb $N=9$ & $k=4$ & $\chi=\one{3}$ & $\Delta_{4,9}=(-\z{6}-1)E_2^{\one{1},\, \ch{9,4}} \cdot E_2^{\, \ch{9,2}, \one{1}}$ \\
\hline
\hline
\rb $N=10$ & $k=6$ & $\chi=\one{10}$ & $\Phi_{6,10}=-6E_2^{\one{1}, \one{10}} \cdot \Phi_{4,10}$ \\
\hline
\hline
\rb $N=15$ & $k=6$ & $\chi=\one{15}$ & $\Phi_{6,15}=-3E_2^{\one{1}, \one{15}} \cdot \Phi_{4,15}$ \\
\hline
\hline
\rb $N=21$ & $k=4$ & $\chi=\one{21}$ & $\Phi_{4,21}=-2E_2^{\one{1}, \one{21}} \cdot \Delta_{2,21}$ \\
\hline
\end{tabular}
\end{table}

\clearpage

\begin{table}[ht]
\caption{Identities with all weights greater than 2, $N>1$} \label{table8}
\begin{tabular}{|c|c|c|l|}
\hline
Level & Weight & Nebentypus & \multicolumn{1} {c|} {Identity}  \\
\hline
\hline
\rb $N=3$ & $k=6$ & $\chi=\one{3}$ & $\Delta_{6,3}=-9E_3^{\one{1},\, \ch{3}} \cdot E_3^{\, \ch{3}, \one{1}}$ \\
\hline
\rb $N=3$ & $k=7$ & $\chi=\ch{3}$ & $\Delta_{7,3,\ch{3}}=-9E_3^{\one{1},\, \ch{3}} \cdot E_4^{\one{3}, \one{1}}$ \\
\hline
\rb $N=3$ & $k=8$ & $\chi=\one{3}$ & $\Delta_{8,3}=-9E_3^{\one{1},\, \ch{3}} \cdot E_5^{\, \ch{3}, \one{1}}$ \\
\hline
\rb $N=3$ & $k=8$ & $\chi=\one{3}$ & $\Delta_{8,3}=3E_5^{\one{1},\, \ch{3}} \cdot E_3^{\, \ch{3}, \one{1}}$ \\
\hline
\rb $N=3$ & $k=10$ & $\chi=\one{3}$ & $\Phi_{10,3}=3E_5^{\one{1},\, \ch{3}} \cdot E_5^{\, \ch{3}, \one{1}}$ \\
\hline
\hline
\rb $N=4$ & $k=6$ & $\chi=\one{2}$ & $\Delta_{6,4}=-4E_3^{\one{1},\, \ch{4}} \cdot E_3^{\, \ch{4}, \one{1}}$ \\
\hline
\rb $N=4$ & $k=8$ & $\chi=\one{2}$ & $\Delta_{8,2}=-4E_3^{\one{1},\, \ch{4}} \cdot \Delta_{5,4,\ch{4}}$ \\
\hline
\hline
\rb $N=5$ & $k=6$ & $\chi=\one{5}$ & $ \Delta_{6,5}=(-2+i)E_3^{\one{1},\, \ch{5,1}} \cdot E_3^{\, \ch{5,3}, \one{1}}$ \\
\hline
\rb $N=5$ & $k=6$ & $\chi=\one{5}$ & $ \Delta_{6,5}=(-2-i)E_3^{\one{1},\, \ch{5,3}} \cdot E_3^{\, \ch{5,1}, \one{1}}$ \\
\hline
\end{tabular}
\end{table}

\bibliographystyle{alpha}
\bibliography{biblio}

\newcommand{\etalchar}[1]{$^{#1}$}
\begin{thebibliography}{Emm05}

\bibitem[BBB{\etalchar{+}}]{PARI}
C.~Batut, K.~Belabas, D.~Bernardi, H.~Cohen, and M.~Olivier.
\newblock {Pari-gp}.
\newblock {\em Available from ftp://megrez. math. u-bordeaux. fr/pub/pari}.

\bibitem[Car59]{carlitz59}
L.~Carlitz.
\newblock {Arithmetic Properties of Generalized Bernoulli Numbers.}
\newblock {\em Journal f{\"u}r die reine und angewandte Mathematik $($Crelles
  Journal$)$}, 1959(202):174--182, 1959.

\bibitem[DI95]{diamondim95}
F.~Diamond and J.~Im.
\newblock {Modular forms and modular curves}.
\newblock In {\em Seminar on Fermat's last theorem: 1993-1994}, pages 39--133.
  Amer Mathematical Society, 1995.

\bibitem[DS05]{diamondshurman05}
F.~Diamond and J.M. Shurman.
\newblock {\em {A first course in modular forms}}.
\newblock Springer Verlag, 2005.

\bibitem[Duk99]{duke99}
W.~Duke.
\newblock {When is the product of two Hecke eigenforms an eigenform?}
\newblock {\em Number theory in progress}, 2:737--741, 1999.

\bibitem[Emm05]{emmons05}
B.A. Emmons.
\newblock {Products of Hecke eigenforms}.
\newblock {\em Journal of Number Theory}, 115(2):381--393, 2005.

\bibitem[Gha00]{ghate00}
E.~Ghate.
\newblock {On monomial relations between Eisenstein series}.
\newblock {\em Ramanujan Mathematical Society}, 15(2):71--80, 2000.

\bibitem[Gha02]{ghate02}
E.~Ghate.
\newblock {On products of eigenforms}.
\newblock {\em Acta Arithmetica}, 102(1):27--44, 2002.

\bibitem[Stu87]{sturm87}
J.~Sturm.
\newblock {On the congruence of modular forms}.
\newblock {\em Lecture Notes in Math}, 1240, 1987.

\bibitem[Was97]{washington97}
L.C. Washington.
\newblock {\em {Introduction to cyclotomic fields}}.
\newblock Springer Verlag, 1997.

\end{thebibliography}
\end{document}